\documentclass[11pt, notitlepage]{article}
\usepackage{amssymb,amsmath,comment}
\catcode`\@=11 \@addtoreset{equation}{section}
\def\thesection{\arabic{section}}

\def\theequation{\thesection.\arabic{equation}}
\catcode`\@=12
\usepackage{colortbl}
\usepackage{a4wide}

\newcommand{\fa} {\forall}
\newcommand{\ds} {\displaystyle}
\newcommand{\e}{\epsilon}
\newcommand{\pa} {\partial}
\newcommand{\al} {\alpha}
\newcommand{\ba} {\beta}
\newcommand{\de} {\delta}
\newcommand{\ga} {\gamma}

\newcommand{\Om} {\Omega}
\newcommand{\ra} {\rightarrow}

\newcommand{\De} {\Delta}
\newcommand{\la} {\lambda}
\newcommand{\La} {\Lambda}
\newcommand{\noi} {\noindent}
\newcommand{\na} {\nabla}

\newcommand{\oline} {\overline}
\newcommand{\mb} {\mathbb}
\newcommand{\mc} {\mathcal}
\newcommand{\lra} {\longrightarrow}
\newcommand{\ld} {\langle}
\newcommand{\rd} {\rangle}

\usepackage[all]{xy}
\catcode`\@=11
\def\theequation{\@arabic{\c@section}.\@arabic{\c@equation}}
\catcode`\@=12

\def\proof{\noindent{\textbf{Proof. }}}
\def\QED{\hfill {$\square$}\goodbreak \medskip}

\newtheorem{Theorem}{Theorem}[section]
\newtheorem{Lemma}[Theorem]{Lemma}
\newtheorem{Proposition}[Theorem]{Proposition}

\newtheorem{Remark}[Theorem]{Remark}
\newtheorem{Definition}[Theorem]{Definition}

\begin{document}

\title
{Polyharmonic Kirchhoff type equations with singular exponential nonlinearities}

\author{
{\bf Pawan Kumar Mishra\footnote{email:pawanmishra31284@gmail.com,}},\;{\bf   \; Sarika Goyal\footnote{email: sarika1.iitd@gmail.com}}\; and {\bf  K. Sreenadh\footnote{e-mail: sreenadh@gmail.com}}\\
{\small Department of Mathematics}, \\{\small Indian Institute of Technology Delhi}\\
{\small Hauz Khaz}, {\small New Delhi-16, India}\\
 }

\date{}

\maketitle

\begin{abstract}

\noi In this article, we study the existence of
non-negative solutions of the following polyharmonic Kirchhoff type problem with critical singular exponential nolinearity
$$ \quad \left\{
\begin{array}{lr}
 \quad  -M\left(\displaystyle\int_\Omega |\nabla^m u|^{\frac{n}{m}}dx\right)\Delta_{\frac{n}{m}}^{m} u = \frac{f(x,u)}{|x|^\alpha} \; \text{in}\;
\Om{,} \\
 \quad \quad u = \nabla u=\cdot\cdot\cdot= {\nabla}^{m-1} u=0 \quad \text{on} \quad \partial \Om{,}
\end{array}
\right.
$$
where $\Om\subset \mb R^n$ is a bounded domain with smooth boundary, $n\geq 2m\geq 2$ and $f(x,u)$ behaves like $e^{|u|^{\frac{n}{n-m}}}$ as $|u|\ra\infty$. Using mountain pass structure and {the} concentration compactness principle, we show the existence of a nontrivial solution. 
In the later part of the paper, we also discuss the above problem with convex-concave type sign changing nonlinearity. Using {the} Nehari manifold technique, we show the existence and multiplicity of non-negative solutions.
\medskip

\noi \textbf{Key words:} Kirchhoff equation,
Trudinger-Moser-Adams embedding, critical singular exponential growth.

\medskip

\noi \textit{2010 Mathematics Subject Classification:} 35J35, 35J60,
35J92

\end{abstract}

\bigskip
\vfill\eject

\section{Introduction}
\setcounter{equation}{0}
Let $\Omega\subset \mb R^n$ is a bounded domain with smooth boundary, $n,m\in \mathbb N$ with $n\geq 2m\geq 2$. We consider the  following polyharmonic Kirchhoff problem:
$$ \mc{( P)}\quad \left\{
\begin{array}{rllll}
 -M\left(\displaystyle\int_{\Om} |\na^m u|^\frac{n}{m}dx\right)\Delta^m_\frac{n}{m} u&=\frac{f(x,u)}{|x|^\alpha} \; \text{in}\;  \Omega,\\
  u=\nabla u=\cdots \nabla^{m-1}u&=0\; \text{on}\;  \partial \Omega,
\end{array}\right.$$
where
$M:\mb R^+\rightarrow \mb R^+$ is a positive, continuous function satisfying $(m1)-(m4)$ stated below and $f:\oline{\Omega}\times\mb R\rightarrow \mb R$ is a Caratheodory function with critical exponential growth satisfying assumptions $(f1)-(f5)$ as stated in section 2.\\
\noi In section 3, we also consider the following problem with sign changing nonlinearity
$$ (\mathcal P_{\la})\quad \left\{
\begin{array}{rllll}
 \quad  {-M\left(\displaystyle \int_\Omega |\nabla^mu|^\frac{n}{m}dx\right)}\De^{m}_{\frac{n}{m}} u &= \la h(x)|u|^{q-1}u+ u|u|^{p} ~ \frac{e^{|u|^{\ba}}}{|x|^\alpha} \; \text{in}\;
\Om{,} \\
  u=\nabla u=\cdots = {\nabla}^{m-1} u &=0\quad\quad \text{on} \quad \partial \Om{,}
 \end{array}
\right.
$$
 where $1< q<\frac{n}{m}<\frac{2n}{m}< p+2$, $0<\alpha<n$, $\ba\in (1,\frac{n}{n-m}]$ and $\la>0$.
\noi For $u\in C^{m}$, the symbol $\na^m u$, where $m$ is a positive integer, denotes the $m^{\text{th}}$ order gradient of $u$ and is defined as,
$$
\na^{m} u=\begin{cases}
\na\De^{(m-1)/2}u & \mbox{if}\; m\; \mbox{is odd},\\
\De^{\frac{m}{2}}u & \mbox{if}\; m\; \mbox{is even}{,}
\end{cases}
$$
where $\na$ and $\De$ denote the usual gradient and Laplacian operator respectively.
The polyharmonic operator $\De^{m}_{\frac{n}{m}} u$ is defined as
$$
 \quad  \De^{m}_{\frac{n}{m}} u :=
\begin{cases}
-\na\{\De^{j-1}(|\na \De^{j-1}|^{\frac{n}{m}-2}\na\De^{j-1}u)\} & \mbox{if}\;\; m=2j-1\\
\De^{j}(|\De^{j}u|^{\frac{n}{m}-2}\De^{j}u) & \mbox{if}\; m=2j
 \end{cases}
$$
for $j=1,2,3,\cdots$.
We assume the following assumptions on $M(t)$:
\begin{enumerate}
  \item[$(m1)$] There exists $M_0>0$ such that $M(t)\geq M_0$ for all $t\geq 0$ and
  \[ \widehat{M}(t+s)\geq \widehat{M}(t)+\widehat M(s)\;\text{for all}\; s,t\geq 0,\]
  where $\widehat M(t)= \int_{0}^{t} M(s)ds$ is the primitive of $M$ such that $\widehat M(0)=0.$
  \item[$(m2)$] There exist constants $a_1$, $a_2>0$ and $t_0>0$ such that for some $\sigma\in\mb R$
  \[  M(t)\leq a_1 + a_2 t^{\sigma}, \; \text{for all}\; t\geq t_0.\]
  \item[$(m3)$] $\frac{M(t)}{t}$ is nonincreasing for $t>0$.
 \end{enumerate}
The condition $(m1)$ is valid whenever $M(0)=M_0$ and $M$ is nondecreasing. A typical example of a function $M$ satisfying the conditions $(m1)-(m3)$ is $M(t)=M_0+at$, where $a\geq 0$. From $(m3)$, we can easily deduce that
\begin{equation}\label{7a0}
 \frac{m}{n}\widehat M(t)-\frac{m}{2n}M(t)t \;\text{ is nondecreasing for }\; t\ge 0.\end{equation}

\noi The above problems are called non-local because of the presence of the term $M\left(\int_\Om |\na^m u|^\frac{n}{m}dx\right)$. Therefore the equation in $(\mc P)$ and $(\mathcal P_\lambda)$ is no longer a pointwise identity. This phenomenon causes some mathematical difficulties which makes the study
of such class of problem interesting. Basically, the presence of $\int_\Om |\na^m u|^\frac{n}{m} dx$ as the coefficient of $\int_{\Om} |\na^m u|^{\frac{n}{m}-2}\na^m u \na^m \phi~dx$ in the weak formulation requires the strong convergence while taking the limit to obtain the weak solution.

\noi The critical exponent problems with exponential type nonlinearities, motivated by Moser-Trudinger inequality \cite{moser}, in the limiting cases was initially studied by Adimurthi \cite{A} and later by several authors \cite{fmr, panda,PS,DONS}. These problems with singular exponential growth nonlinearities for {Laplacian and $n$-Laplacian} was studied in \cite{AS}, where an interpolation inequality of Hardy and Moser-Trudinger inequality (see \cite{AY} also) is proved for $W^{1,n}_{0}(\Om)$. Subsequently, this inequality is generalized to higher order spaces in \cite{SMA}, known as singular Adam's-Moser inequality:
\begin{Theorem}\label{admo}
 For $0<\alpha<n$ and $\Omega$ be a bounded domain in $\mathbb R^n $. Then for all $0 \leq \nu \leq \beta_{\alpha}=\left(1-\frac{\alpha}{n}\right)\beta_{n,m}$, we have
\begin{align}\label{eqc002}
\displaystyle\sup_{u\in W^{m,\frac{n}{m}}_{0}(\Om),\| u\| \leq 1}\int_{\Om} \frac{e^{\nu |u|^{\frac{n}{n-m}}}}{|x|^\alpha} dx
<\infty{,}
\end{align}
where
$$\ba_{m,n}= \left\{
\begin{array}{lr}
\frac{n}{w_{n-1}}\left(\frac{\pi^{n/2}2^{m}\Gamma(\frac{m+1}{2})}{\Gamma(\frac{n-m+1}{2})}\right)^{\frac{n}{n-m}}\;\mbox{when}\; m \; \mbox{is odd}\\
  \frac{n}{w_{n-1}}\left(\frac{\pi^{n/2}2^{m}\Gamma(\frac{m}{2})}{\Gamma(\frac{n-m}{2})}\right)^{\frac{n}{n-m}}\;\;\;\;\mbox{when}\; m \; \mbox{is even}.
  \end{array}
  \right.
 $$
and $w_{n-1}=$ volume of $\mb S^{n-1}$.
\end{Theorem}
The above inequality is sharp in the sense that if $\nu>\beta_{\alpha}$, then the supremum in \eqref{eqc002} is infinite.
The embedding $W_{0}^{m,\frac{n}{m}}(\Om) \ni u\longmapsto \frac{e^{|u|^{\ba}}}{|x|^{\al}} \in
L^{1}(\Om)$ is compact for all $\ba\in\left(1,\beta_{\al}\right)$
and is continuous for $\ba=\ba_{\al}$. The non-compactness of
the embedding can be shown using a sequence of functions that are obtained from the fundamental solution of $-\De^{m}_{\frac{n}{m}}$. The above embedding in case of the whole space $\mathbb R^n$ is studied {in} \cite{ngf}.
The existence results for quasilinear polyharmonic  problems
with exponential terms on bounded domains is studied in \cite{SMA}.\\

\noi The polyharmonic problems for critical exponents have been studied by many authors, see \cite{flgz,pucciserin,hcgr,YGJ} and references therein. In \cite{pucciserin}, authors have considered the following polyharmonic problems with Sobolev critical exponent in a open ball of $\mathbb R^n$
 \begin{equation*}
 (-\Delta)^Ku=\lambda u+|u|^{s-1}u,\;\; \textrm{in} \; B,\; u=Du=.....D^{K-1}u=0\;\textrm{on}\; \partial B{,}
 \end{equation*}
 where $K-1\in \mathbb N$, $s=\frac{n+2K}{n-2K}$ and $n>2K$. Here authors have shown the existence of nontrivial radial solution for suitable range of $\lambda$. In \cite{flgz}, authors have studied the polyharmonic problems for Sobolev critical growth nonlinearity in a bounded domain and shown the existence of {a} nontrivial solution.
 The existence results for polyharmonic equations with exponential growth nonlinearity have also been discussed by many authors, see \cite{ng, ol, sph, zd} and references therein. \\

\noindent The boundary value problems involving Kirchhoff equations arise in several physical and biological systems. This type of non-local problems were initially observed by Kirchhoff in 1883  in the study of string or membrane vibrations to describe the transversal oscillations of a stretched string by taking into account the subsequent change in string length caused by oscillations.  In the Laplacian case the problem of the above type arises from the theory of thin plates
and describes the deflection of the middle surface of a $p$-power-like elastic isotropic flat plate of uniform thickness, with nonlocal flexural rigidity of
the plate $M(\|u\|^p)$ depending continuously on $\|u\|^p$ of the deflection u and subject to nonlinear source forces.\\

\noi The nonlocal Kirchhoff problems  with Sobolev type critical nonlinearities is initially studied in \cite{ACF}.
In \cite{ACF}, authors considered the following critical Kirchhoff problem
\begin{align*}
-M(\int_\Om |\nabla u|^2dx ) \De u = \la f(x,u) + u^5 \; \text{in}\; \Om, \;\; u>0 \; \text{on}\; \pa \Om,
\end{align*}
where $\Omega \subset \mathbb{R}^3$ is a bounded domain with smooth boundary and $f$ has subcritical growth at $\infty$. Using the mountain-pass lemma and compactness analysis of local Palais-Smale sequences authors showed the existence of solutions for large $\la$. Existence of positive solutions for $p$-Kirchhoff equations with super critical terms is studied in \cite{FG}.  { The critical Kirchhoff problem in bounded domains in $\mathbb{R}^2$ with exponential growth nonlinearity is studied in \cite{gs,Racsam}.\\

\noi The polyharmonic Kirchhoff problems for subcritical exponents are studied in \cite{apuc}. In \cite{apuc}, authors have considered the following Kirchhoff type problem for higher dimensions
\begin{equation*}
M(\|u\|^p)\Delta^K_pu=f(x,u)\;\;\textrm{in}\; \Omega,\;\; D^\alpha u=0 \; \text{on}\; \partial \Omega
\end{equation*} for each multi-index $\alpha$ such that $|\alpha|\leq K-1$ with $f$ having subcritical growth assumptions
 and studied existence of multiple solutions via three
critical points theorem given in \cite{fcppv}.  Here, authors have also extended the results to the $p(x)$-polyharmonic Kirchhoff problems.\\

\noi The multiplicity of positive solutions for elliptic equations was initiated in \cite{TA} and later a global multiplicity result was obtained in \cite{ABC}. Later,  lot of works were devoted to
address the multiplicity results for
quasilinear elliptic problems with positive nonlinearities.
Also,  many authors studied  these multiplicity
results with polynomial type nonlinearity with sign-changing weight
functions using the Nehari manifold and fibering map analysis (see \cite{TA,DP,WU,WU10,COA,KY}). In \cite{CK}, authors studied the existence of multiple positive solution of Kirchhoff type problem with convex-concave polynomial type nonlinearities having subcritical growth by Nehari manifold and fibering map methods. \\

\noindent In the first part of the paper, we discuss the existence result for the polyharmonic Kirchhoff problem  with nonlinearity $f(x,u)$ that has  superlinear growth near zero and exponential growth near $\infty$. To prove our results, we study the first critical level and study the Palais-Smale sequences below this level. Using the concentration compactness principle we show the existence of a subsequence, of Palais-Smale sequences, for which $\{|\na^m u_k|^{\frac{n}{m}-2}\na^m u_k\}$ converges weakly in $ L^{\frac{n}{n-m}}(\Om)$. This combined with singular version of Lion's Lemma on higher integrability is used to show the existence of a nontrivial solution. In the later part of the paper, we discuss existence and multiplicity for sign changing nonlinearity using Nehari manifold and fibering map analysis. { In  the case $M\equiv 1$,  $\alpha=0$,  in \cite{sgana}, authors have considered the critical exponent problem and proved the existence of non-trivial solution. They obtained the solution as weak limit of Palais-Smale sequence below the first critical level. But when $M\not\equiv 1$, weak convergence is not enough to obtain the solution. In order to overcome this, we study the concentration phenomena of Palais-Smale sequences to obtain a convergent subsequence.}

\noi We use the notation $|\na^m u|$ to denote the Euclidean length of the vector $\na^m u$. We will also use the notations $\|u\|_{\frac{n}{m}}$ and $\|u\|$ for the $L^{\frac{n}{m}}(\Om)$  and $W^{m,\frac{n}{m}}(\Om)$ norms of $u$ respectively. Also we have the notation $L^s(\Omega, |x|^{-\alpha}dx)$ for weighted Lebesgue spaces with measure $|x|^{-\alpha}dx$.  The weak convergence is denoted by
$\rightharpoonup$ and $\ra$ denotes strong convergence.

\section{Critical exponent problem with  positive nonlinearity}
\setcounter{equation}{0}
\noi In this section, we consider the problem
$$ \mc{( P)}\quad \left\{
\begin{array}{rllll}
 -M\left(\displaystyle\int_{\Om} |\na^m u|^\frac{n}{m}dx\right)\Delta^m_\frac{n}{m} u&=\frac{f(x,u)}{|x|^\alpha} \; \text{in}\;  \Omega,\\
  u=\nabla u=......\nabla^{m-1}u&=0\; \text{on}\;  \partial \Omega,
\end{array}\right.$$
where $\Omega\subset \mb R^n$ is a bounded domain with smooth boundary {and $n\geq 2m\geq 2, n,m\in \mathbb N$.}
The nonlinearity $f(x,t)=h(x,t) e^{|t|^{n/n-m}}$, where $h(x,t)$ satisfies
\begin{enumerate}
\item[$(f1)$] $h\in C^1(\overline{\Om}\times \mb R)$, $h(x,0)=0,$ for all $t\le 0$, $h(x,t)>0,$ \text{for all} $t>0.$
\item[$(f2)$] For any $\e>0,$ $\ds \lim_{t\ra \infty}\sup_{x\in \overline{\Om}} h(x,t) e^{-\e |t|^{n/n-m} }=0$, $\ds\lim_{t\ra \infty}\inf_{x\in \overline{\Om}} h(x,t) e^{\e|t|^{n/n-m}}=\infty.$
\item[$(f3)$] There exist positive  constants $t_0$, $K_0>0$ such that
\[  F(x,t)\le K_0 f(x,t)\;\mbox{for all}\; (x,t)\in \Om\times[t_0,+\infty),\]
{where $F(x,t)=\int_0^t f(x,s)ds$ is the $t$-primitive of $f(x,t)$.\\}
\item[$(f4)$] For each $\ds x\in \Omega, \frac{f(x,t)}{t^{\frac{2n-m}{m}}}$ is increasing for $t>0$ and $\ds\lim_{t\rightarrow 0^+} \frac{f(x,t)}{t^{\frac{2n-m}{m}}}=0.$
\item[$(f5)$] $\ds \lim_{t\rightarrow \infty} t h(x,t)=\infty.$
\end{enumerate}
\noi Assumption $(f3)$ implies that $\ds  F(x,t)\ge F(x,t_0)e^{K_0(t-t_0)}$, for all $(x,t)\in \Omega \times [t_0, \infty)$ which is a reasonable condition for function behaving as $e^{b |t|^{n/n-m}}$ at $\infty.$ Moreover from $(f3)$ it follows that for each $\theta>0, $ there exists $R_\theta>0$ satisfying
\begin{equation}\label{n7a0}
\theta F(x,t)\le t f(x,t)\; \text{for all}\; (x,t)\in \Om \times [R_\theta, \infty).
\end{equation}
We also have that condition $(f4)$ implies that for $\mu \in [0,\frac{2n-m}{m})$,
\begin{equation}\label{n1}
\lim_{t\rightarrow 0^+} \frac{f(x,t)}{t^\mu}=0, \; \text{uniformly in }\; x\in \Om.
\end{equation}
\noi In this section we define the variational setting for the problem and study the existence of mountain pass solution of $(\mathcal P)$.
\noi Generally, the main difficulty encountered in non-local Kirchhoff problems is the competition between the growths of $M$ and $f$. {Here we generalize the result of \cite{sph} to the polyharmonic Kirchhoff equation with singular exponential growth using singular Adams-Moser inequality \eqref{eqc002}.}
\begin{Definition}
We say that $u\in W^{m,\frac{n}{m}}_{0}(\Om)$ is a weak solution of
$(\mathcal P)$ if for all $\phi \in W^{m,\frac{n}{m}}_{0}(\Om)$, we have
\begin{align*}
M\left(\int_\Omega |\nabla^m u|^{\frac{n}{m}}dx\right)\int_{\Om}|{\na}^{m} u|^{\frac{n}{m}-2}{\na}^{m} u {\na}^{m} \phi ~dx =  \int_{\Om}\frac{f(x,u)}{|x|^\alpha}\phi ~dx.
\end{align*}
\end{Definition}

\noi The energy functional associated with the problem $(\mathcal P)$ is $\mathcal J :W^{m,\frac{n}{m}}_{0}(\Om) \lra \mb{R}$ defined as
\begin{equation}\label{eq1}
\mathcal J(u)= \frac{m}{n}\widehat M\left(\int_\Omega |\nabla^m u|^{\frac{n}{m}}dx\right) -  \int_{\Om}\frac{F(x,u)}{|x|^\alpha}~ dx{,}
\end{equation}
The functional $\mathcal J$ is differentiable on $W^{m,\frac{n}{m}}_{0}(\Om)$ and for all $ \phi\in W^{m,\frac{n}{m}}_{0}(\Om)$,
\[\langle \mathcal J'(u),\phi\rangle= M\left(\int_\Omega |\nabla^m u|^{\frac{n}{m}}dx\right)\int_{\Om}|{\na}^{m} u|^{\frac{n}{m}-2}{\na}^{m} u {\na}^{m} \phi ~dx -  \int_{\Om}\frac{f(x,u)}{|x|^\alpha}\phi ~dx.\]
Now we state the existence result for the problem $(\mathcal P)$ in the following {Theorem}.
\begin{Theorem}\label{main}
If $f$ satisfies $(m1)-(m5)$ and $(f1)-(f5)$ then the problem $(\mathcal P)$ admits a nontrivial solution.
\end{Theorem}
\noi We prove this Theorem by mountain pass Lemma. In the following Lemmas we study the mountain pass structure of the functional $\mathcal J$.
\begin{Lemma} There exists $R_0>0$ and $\eta>0$ such that $\mathcal J(u)\geq \eta$ for all $\|u\|=R_0$.
\end{Lemma}
\begin{proof} From the assumptions, $(f1)-(f3)$, for $\e>0$, $r>\frac{n}{m}$, there exists $C>0$ such that
\[|F(x,t)| \le \e |t|^\frac{n}{m} + C |t|^r e^{|t|^{n/n-m}},\;\; \text{for all}\; (x,t)\in \Omega \times \mb R.\]
Now,
using continuous imbedding of $W_0^{m, \frac{n}{m}}(\Omega)\hookrightarrow L^s(\Omega; |x|^{-\alpha}dx)$ for all $s\in [1, \infty)$, we get
\begin{equation}\label{ione}
\int_\Om \frac{|u|^\frac{n}{m}}{|x|^\alpha} dx\leq C_1\|u\|^{\frac{n}{m}}\; \text{for all} \; u\in W_{0}^{m,\frac{n}{m}}.
\end{equation}
Similarly, using \eqref{ione}, we get
\begin{align*}
\int_\Om {|u|^r e^{|u|^{n/n-m}}}{|x|^{-\alpha}} dx\leq C \left(\int_\Om {e^{p|u|^{n/n-m}}}{|x|^{-\alpha}}\right)^\frac{1}{p}
\leq C\left(\int_\Om {e^{p\left(\|u\| \frac{|u|}{\|u\|}\right)^{\frac{n}{n-m}}}}{|x|^{-\alpha}}\right)^\frac{1}{p}{.}\\
\end{align*}
Now choose $\|u\|=R_0>0$ such that $p\|u\|^{\frac{n}{n-m}}<\beta_{\alpha}$, and using Theorem \ref{admo},
we get
\begin{equation}\label{itwo}
\int_\Om {|u|^r e^{|u|^{n/n-m}}}{|x|^{-\alpha}} dx\leq C_2\|u\|^r{.}
\end{equation}
Also from \eqref{ione} and \eqref{itwo}, we have
\begin{align*}
\int_\Om \frac{F(x,u)}{|x|^\alpha} dx \le \e\int_\Om {|u|^\frac{n}{m}}{|x|^{-\alpha}} dx +C \int_\Om {|u|^r e^{|u|^{n/n-1}}}{|x|^{-\alpha}} dx
\le \e C_1 \|u\|^\frac{n}{m} + C_2 \|u\|^{r}{.}
\end{align*}
Hence
\[\mathcal J(u) \ge \|u\|^\frac{n}{m} \left(M_0\frac{m}{n}-\e C_1 - C_2 \|u\|^{r-\frac{n}{m}}\right).\]
Since $r>\frac{n}{m},$ we can choose $\e$ and $R_0$ such  that $J(u)\ge \tau$ for some $\tau$ on $\|u\|=R_0$.\QED
\end{proof}
\begin{Lemma}
There exists $e\in W_0^{m, \frac{n}{m}}(\Omega)$ such that $\mathcal{J}(e)<0$.
\end{Lemma}
\begin{proof}
\noi From equation $\eqref{n7a0}$, for $\theta>0$, there exist $C_1$, $C_2>0$ such that
 \begin{equation*}
 F(x,t)\geq C_1 t^{\theta}- C_2\;\mbox{for all}\; (x,t)\in \Om\times[0,+\infty){.}
 \end{equation*}
 Now from assumption $(m2)$, for all $t\geq t_0$
\begin{equation}\label{n2}
\widehat M(t)\leq\left\{
 \begin{array}{lr}
 a_0+a_1t+\frac{a_2}{\sigma+1} t^{\sigma+1}\; \mbox{if}\; \sigma\ne -1,\\
 b_0+ a_1 t+a_2 \ln t\quad\quad\mbox{if}\; \sigma=-1,
 \end{array}
 \right.
\end{equation}
where $a_0= \widehat M(t_0) -a_1t_0-a_2 \frac{t_0^{\sigma+1}}{(\sigma+1)}$ and $b_0= \widehat M(t_0) -a_1 t_0-a_2\ln t_0$.
Now, choose a function $\phi_0 \in W^{m,\frac{n}{m}}_{0}(\Om)$ with $\phi_0\geq 0$ and $\|\phi_0\|=1$. Then from \eqref{eq1} and \eqref{n2}, for all $t\geq t_0$, we obtain
\begin{equation*}
\mathcal J(t\phi_0)\leq \left\{
\begin{array}{lr}
\frac{m}{n}a_0+\frac{m }{n}a_1 t^\frac{n}{m}+\frac{m}{n}\frac{a_2}{\sigma+1} t^{\frac{n}{m}(\sigma+1)}- C_1 t^{\theta}\displaystyle\int_\Omega \frac{|\phi_0|^{\theta}}{|x|^\alpha}dx+C_2\displaystyle\int_\Omega\frac{1}{|x|^\alpha}dx,\; \mbox{if}\; \sigma\ne -1,\\
\frac{ m}{n}b_0+ \frac{m}{n}a_1 t^\frac{n}{m} +a_2 \ln t - C_1 t^{\theta}\displaystyle\int_\Omega \frac{|\phi_0|^{\theta}}{|x|^\alpha}dx+C_2\displaystyle\int_\Omega\frac{1}{|x|^\alpha}dx\;\;\;\quad\quad\mbox{if}\; \sigma=-1{.}
\end{array}
\right.
\end{equation*}
Now choose $\theta>\max\{\frac{n}{m},\frac{n}{m}(\sigma+1)\}$ we conclude that $\mathcal J(t \phi_0)\ra -\infty$ as $t\ra +\infty$.
Therefore, there exists  $e\in W_0^{m, \frac{n}{m}}(\Omega)$ such that $\mathcal{J}(e)<0$.\QED
\end{proof}
\subsection{Analysis of Palais-Smale sequence}
\noi In this section we study the first critical level of $\mathcal J$ and study the convergence of Palais--Smale sequences below this level.
\begin{Lemma}\label{bps}
Every Palais-Smale sequence of $\mathcal J$ is bounded in $W^{m,\frac{n}{m}}_{0}(\Om)$.
\end{Lemma}
{\proof} Let $\{u_k\}\subset W^{m,\frac{n}{m}}_{0}(\Om)$ be a Palais-Smale sequence for $\mathcal J$ at level $c$, that is
\begin{equation}\label{n7a1}
\frac{m}{n}\widehat M(\|u_k\|^\frac{n}{m})-\int_\Om \frac{F(x,u_k)}{|x|^\alpha}dx \rightarrow c
\end{equation}
and for all $\phi \in W^{m,\frac{n}{m}}_{0}(\Om)$
\begin{equation}\label{n7a2}
 \left|M(\|u_k\|^\frac{n}{m})\int_\Om |\na^m u_k|^{\frac{n}{m}-2} \na^m u_k \na^m \phi~ dx -\int_\Om \frac{f(x,u_k)}{|x|^\alpha} \phi~ dx\right| \le \e_k \|\phi\|{,}
\end{equation}
where $\e_k\rightarrow 0$ as $k\rightarrow \infty.$
From \eqref{7a0}, \eqref{n7a0}, \eqref{n7a1} and \eqref{n7a2}, we obtain
 \begin{align*}
c+ \epsilon_k\|u_k\|&\ge \frac{m}{n}\widehat M(\|u_k\|^\frac{n}{m})-\frac{1}{\theta} M(\|u_k\|^\frac{n}{m}) \|u_k\|^\frac{n}{m} -\int_\Om \frac{F(x,u_k)-\frac{1}{\theta}f(x,u_k)u_k}{|x|^\alpha}dx \\
&\geq \left(\frac{m}{2n}-\frac{1}{\theta}\right) M(\|u_k\|^\frac{n}{m}) \|u_k\|^\frac{n}{m}.
\end{align*}
From this and taking $\theta >\frac{2n}{m}$, we obtain the boundedness of the sequence.  \QED
\noindent Now we construct a sequence of functions, using Adam's function \cite{DRA}.
Let $\mb B$ denote the unit ball $\mb B(0,1)$ with center $0$ in $\mb R^{n}$ and $\mb B_{l}:= \mb B(0,l)$.
We have the following results (see \cite{ol}).
For all $l\in(0,1)$, there exists
\begin{align}\label{ec1}
U_{l}\in \{u\in W^{m,\frac{n}{m}}_{0}(\mb B): u|_{\mb B_l}=1\}
\end{align}
such that
\[\|U_l\|^{\frac{n}{m}}= C_{m,\frac{n}{m}}(\mb B_{l}; \mb B)\leq \left(\frac{m\ba_{\alpha}}{n \log(\frac{1}{l})}\right)^{\frac{n-m}{m}}{,}\]
{where} $C_{m,\frac{n}{m}}(K;E)$ denotes the $(m,\frac{n}{m})-$conductor capacity of $K$ in $E$, whenever $E$ is an open set and $K$ a relatively compact subset of $E$. It is defined as
\[C_{m,\frac{n}{m}}(K;E):= \inf\left\{\|u\|^{\frac{n}{m}}: u\in C_{0}^{\infty}(E), u|_{K} =1\right\}.\]

\noi Now, let $x_0\in \Om$, $R\leq R_{0}= d(x_0, \partial \Om)$, where $d$ denotes the distance from $x_0$ to $\partial\Omega$. Then the Adams function $\tilde{A}_r$ is defined as
$$\tilde{A}_{r}(x)=\begin{cases}
\left(\frac{n \log(\frac{R}{r})}{m\ba_{\alpha}}\right)^{\frac{n-m}{n}}\; U_{\frac{r}{R}}(\frac{x-x_0}{R}) & \mbox{if}\;\; |x-x_0|<R\\
0 & \mbox{if}\;\; |x-x_0|\geq R,
\end{cases}
$$
where $0<r<R$ and $U_{l}$ is as in equation \eqref{ec1}. Also it is easy to check that $\|\tilde{A}_{r}\|\leq 1$.

\noi Firstly, Let $\de_k>0$ be such that $\de_k \ra 0$ as $k\ra\infty$. We will choose the exact choice of $\de_k$ later. Now by taking $x_0=0$, $R= \de_k$ and $r=\frac{\de_k}{k}$, we define
$${A}_{k}(x)=\begin{cases}
\left(\frac{n \log k}{m\ba_{\alpha}}\right)^{\frac{n-m}{n}}\; U_{\frac 1k}(\frac{x}{\de_k}), & \mbox{if}\; \;|x|<\de_k\\
 0, &\mbox{if}\; |x|\geq \de_k,
\end{cases}
$$
where $A_{k}=\tilde{A}_{r}$. Then $A_{k}(0)= \left(\frac{n \log k}{m\ba_{\alpha}}\right)^{\frac{n-m}{n}}$ and $\|A_k\|\leq 1$.
\noi Let $\ds \Gamma=\{\gamma\in C([0,1],W_0^{m,\frac{n}{m}}(\Omega)):\gamma(0)=0,\mathcal J(\gamma(1))<0\}$ and define the mountain-pass level
$\ds c_*=\inf_{\gamma\in \Gamma}\max_{t\in[0,1]}\mathcal J(\gamma(t))$. Then we have,
\begin{Lemma}\label{le1}
There exists $k$ with $A_k\in W_{0}^{m,\frac{n}{m}}(\Om)$, $\|A_k\|\leq 1$ such that
$c_*< \frac{m}{n} \widehat M\left(\ba_{\alpha}^{\frac{n-m}{m}}\right)${.}
\end{Lemma}

\begin{proof}
Suppose for the sake of contradiction that for all $k$, we have
\[c_*\geq  \frac{m}{n}\widehat M\left(\ba_{\alpha}^{\frac{n-m}{m}}\right).\]
Then for all $k$, there exist $t_k>0$ such that
$\mathcal J(t_k A_k)\geq \frac{m}{n}\widehat M\left(\ba_{\alpha}^{\frac{n-m}{m}}\right).$ Therefore

\[\frac{m}{n} \widehat M(\|t_kA_k\|^\frac{n}{m})\geq\frac{m}{n}\widehat M\left(\ba_{\alpha}^{\frac{n-m}{m}}\right){.}
\]
Now using the assumption $(m1)$, we get
\begin{align}\label{11}
t_{k}^{\frac{n}{m}}\geq {\ba_{\alpha}}^{\frac{n-m}{m}}.
\end{align}
Now, since $\frac{d}{dt}\mathcal J(t A_{k})=0$ at $t=t_k$,  it follows that
\begin{align*}
M\left(t_{k}^{\frac{n}{m}}\|A_k\|^{\frac{n}{m}}\right) t_{k}^{\frac{n}{m}}\|A_k\|^\frac{n}{m}=\int_{\Om} t_{k}A_{k} \frac{f(x, t_k A_k)}{|x|^\alpha}dx.
\end{align*}
\noi Using equation \eqref{11}, we estimate
\begin{align*}
\int_{\Om} t_k A_k \frac{f(x,t_k A_k)}{|x|^\alpha} dx &\geq \int_{|x|\leq \frac{\de_k}{k}} t_k A_k \frac{f(x,t_k A_k)}{|x|^\alpha}~ dx\\
&=  t_k A_k(0) h(x_0,t_k A_k(0)) e^{|t_k A_k(0)|^{\frac{n}{n-m}}}\frac{w_{n-1}}{n-\alpha}\left(\frac{\de_k}{k}\right)^{n-\alpha}\\
&\geq t_k A_{k}(0) h(x_0,t_k A_k(0)) k^{n\frac{t_k^{\frac{n}{n-m}}}{\beta_{\alpha}}}\left(\frac{\de_k}{k}\right)^{n-\alpha}\\
 & \geq C t_k A_{k}(0) h(x_0,t_k A_{k}(0)) k^\alpha \de_{k}^{n-\alpha}
 \end{align*}
for some constant $C>0$. Now we choose $\delta_k=(\log k)^\frac{-1}{\eta(n-\alpha)}$ with $\eta>\frac{n}{n-m}$ in the expression for $A_k(0)$ to get
 \begin{equation}\label{firstlevel}
 \int_{\Om} t_k A_k \frac{f(x,t_k A_k)}{|x|^\alpha} dx \geq C \frac{w_{n-1}}{n-\alpha}\left(\frac{n}{\beta_{\alpha}}\right)^{\frac{n-m}{m}} t_k h(x_0,t_k A_{k}(0))(\log k)^{\frac{n-m}{m}-\frac{1}{\eta}}{.}
 \end{equation}
 From $(f5)$, we see that the right hand side of \eqref{firstlevel} tends to $\infty$ as $k\rightarrow \infty.$
Thus $t_{k}^{\frac{n}{m}}\ra\infty$ as $k\ra\infty$ which is a contradiction with a bounded $t_k$. This completes the proof of Lemma.\QED
\end{proof}

\subsection{Existence of Solution}
In this subsection, we study the first critical level and study the Palais--Smale sequences below this level. Using the concentration compactness principle of Lions \cite{li}, we show that the Palais--Smale sequence $\{u_k\}$ below the first critical level satisfies the property that $\{|\nabla^{m}u_k|^{\frac{n}{m}-2}\nabla^m u_k\}$ converges weakly in $L^{\frac{n}{n-m}}(\Om)$ (see Lemma \ref{IC2}). We show that weak limit of Palais-Smale sequence is the required solution of the problem $(\mathcal P)$.
\noi In order to prove that a Palais-Smale sequence converges to a solution of problem ($\mathcal P$) we need  the following  convergence Lemma. We refer to Lemma 2.1 in \cite{fmr} for a  proof.
\begin{Lemma}\label{lc1}
Let $\Om\subset \mb R^n$ be a bounded domain and $f:\overline{\Om}\times \mb R \ra \mb R$ a continuous function. Then for any sequence $\{u_k\}$ in
$L^{1}(\Om)$ such that
\[u_k\ra u\;\mbox{in}\; L^{1}(\Om),\quad \frac{f(x,u_k)}{|x|^\alpha}\in L^{1}(\Om)\; \mbox{and}\; \int_{\Om}\frac{|f(x,u_k)u_k|}{|x|^\alpha}dx\leq C,\]
we have up to a subsequence $\frac{f(x,u_k)}{|x|^\alpha}\ra \frac{f(x,u)}{|x|^\alpha}$ and $\frac{F(x,u_k)}{|x|^\alpha}\ra \frac{F(x,u)}{|x|^\alpha}$ strongly in $L^{1}(\Om)$.\QED
\end{Lemma}

\noi Now we need the following Lemma, inspired by \cite{do6}, to show that weak limit of a Palais--Smale sequence is a weak solution of $(\mathcal P)$.

\begin{Lemma}\label{IC2}
 For any Palais--Smale sequence $\{u_k\}$, there exists a subsequence
 denoted by $\{u_k\}$ and $u\in W^{m,\frac{n}{m}}_{0}(\Om)$ such that
\[\frac{f(x,u_k)}{|x|^\alpha}\ra \frac{f(x,u)}{|x|^\alpha}\;\mbox{in}\;L^{1}(\Om)\]
\begin{align}\label{ee3}
|\na^m u_k|^{\frac{n}{m}-2}\na^m u_k \rightharpoonup |\na^m u|^{\frac{n}{m}-2}\na^m u\;\; & \mbox{weakly in}\;\; L^{\frac{n}{n-m}}(\Om) \;\text{if}\; m \; \text{is even},\\
|\na^m u_k|^{\frac{n}{m}-2}\na^m u_k \rightharpoonup |\na^m
u|^{\frac{n}{m}-2}\na^m u\; & \mbox{ weakly in}\;\;
(L^{\frac{n}{n-m}}(\Om))^n \;\text{if}\; m \;\text{is odd}.\nonumber
\end{align}
\end{Lemma}

\begin{proof}
From Lemma \ref{bps}, we obtain that  $\{u_k\}$ is bounded in $W^{m,\frac{n}{m}}_{0}(\Om)$. Consequently,  up to a subsequence
$u_k\rightharpoonup u\;\mbox{ weakly in}\; W^{m,\frac{n}{m}}_{0}(\Om)$, $u_k\ra u\;\mbox{strongly in}\; L^{q}(\Om)\;\mbox{for all}\; q\in[1,\infty)$ and $
u_{k}(x)\ra u(x)\; \mbox{a.e in}\;\Om$. Then using the fact that $\{u_k\}$ is a bounded sequence together with equations \eqref{n7a1}, \eqref{n7a2}  and by Lemma \ref{lc1}, we obtain $\frac{f(x,u_k)}{|x|^\alpha}\ra \frac{f(x,u)}{|x|^\alpha}$ in $L^{1}(\Om)$.\\

\noi Claim 1: $\{u_k\}$ has a subsequence such that \eqref{ee3} holds.
\noi Indeed, since $\left\{|\na^m u_k|^{\frac{n}{m}-2}\na^m u_k\right\}$ is bounded in $L^{\frac{n}{n-m}}(\Om)$. Then, without loss of generality, we may assume that
\[|\na^m u_k|^{\frac{n}{m}} \lra \mu\;\mbox{in}\; D^{\prime}(\Om)\;\mbox{and} \]
\[|\na^m u_k|^{\frac{n}{m}-2}\na^{m}u_k \rightharpoonup \nu \; \mbox{weakly in}\; L^{\frac{n}{n-m}}(\Om),\]
\noi where $\mu$ is a non-negative regular measure and $D^{\prime}(\Om)$ are the distributions on $\Om$.

\noi Let $\sigma>0$ and $\mc A_{\sigma}= \{x\in\overline{\Om}: \fa r>0, \mu(B_{r}(x)\cap \overline{\Om})\geq \sigma\}$. We claim that $A_{\sigma}$ is a finite set.
Suppose by contradiction that there exists a sequence of distinct points $(x_s)$ in $\mc A_{\sigma}$. Since for all $r>0$,
$\mu(B_{r}(x)\cap \overline{\Om})\geq \sigma$, we have that $\mu(\{x_s\})\geq \sigma$. This implies that $\mu(\mc A_{\sigma})= +\infty$, however
\[\mu(\mc A_{\sigma})= \lim_{k\ra+\infty} \int_{\mc A_{\sigma}}|\na^{m} u_k|^{\frac{n}{m}} dx \leq C.\]
Thus $\mc A_{\sigma}=\{x_1, x_2,\cdots, x_p\}$.

\noi Let $u\in W^{m,\frac{n}{m}}_{0}(\Omega)$. We know that there exists positive constants $\beta_\alpha$, $C_2$ depending on $n$, $m$ such that
\[\int_{\Omega} \frac{e^{\beta_\alpha(\frac{|u|}{\|\na^m u\|_{{\frac{n}{m}}}})^{\frac{n}{n-m}}}}{|x|^\alpha} dx \leq C_2.\]

\noi {\bf Assertion 1}. If we choose $\sigma>0$ such that
$\sigma^{\frac{m}{n-m}}<\beta_\alpha$, then we have
\[\lim_{k\ra \infty}\int_{K} \frac{f(x,u_k)u_k}{|x|^\alpha}~ dx= \int_{K} \frac{f(x,u)u}{{|x|^\alpha}}~ dx,\]
for any relative compact subset $K$ of $\overline{\Om} \setminus \mc A_{\sigma}$.\\
Indeed, let $x_0\in K$ and
$r_0>0$ be such that $\mu(B_{r_0}(x_0)\cap \oline{\Om})<\sigma$. Consider a function $\phi\in C^{\infty}_{0}(\Om,[0,1])$ such that $\phi\equiv1$ in $B_{\frac{r_0}{2}}(x_0)\cap \overline{\Om}$ and $\phi\equiv0$ in $\overline{\Om}\setminus (B_{r_0}(x_0)\cap\oline{\Om})$.
Thus
\[\lim_{k\ra\infty} \int_{B_{r_0}(x_0)\cap \overline{\Om}}|\na^m u_k|^{\frac{n}{m}} \phi ~dx = \int_{B_{r_0}(x_0)\cap \overline{\Om}}\phi d\mu \leq \mu(B_{r_0}(x_0)\cap \overline{\Om})<\sigma. \]
Therefore for $k\in \mb N$ sufficiently large  and $\e>0$ sufficiently small, we have
\[ \int_{B_{\frac{r_0}{2}}(x_0)\cap \overline{\Om}}|\na^m u_k|^{\frac{n}{m}} ~dx = \int_{B_{\frac{r_0}{2}}(x_0)\cap \overline{\Om}}|\na^m u_k|^{\frac{n}{m}}\phi dx \leq (1-\e)\sigma,\]
which together with Theorem \ref{admo} implies
{\small \begin{align}\label{e9}
\int_{B_{\frac{r_0}{2}}(x_0)\cap \overline{\Om}} \left(\frac{|f(x,u_k)|}{|x|^\alpha}\right)^q dx &=\int_{B_{\frac{r_0}{2}}(x_0)\cap \overline{\Om}}|h(x,u_k)|^{q} \frac{e^{q|u_k|^{\frac{n}{n-m}}}}{|x|^{q\alpha}} dx\notag\\
 &\leq d \int_{B_{\frac{r_0}{2}}(x_0)\cap \overline{\Om}} \frac{e^{(1+\de)q|u_k|^{\frac{n}{n-m}}}}{|x|^{q\alpha}}dx\leq K
\end{align}}
if we choose $q>1$ sufficiently close to $1$ and $\de>0$ is small enough such that
$\frac{(1+\de)q \sigma^{\frac{m}{n-m}}}{\beta_\alpha}<1$. Now we estimate
\[\int_{B_{\frac{r_0}{2}}(x_0)\cap \overline{\Om}}\frac{|f(x,u_k)u_k-f(x,u)u|}{{|x|^\alpha}} ~dx \leq I_1+ I_2\]
where \[I_1 = \int_{B_{\frac{r_0}{2}}(x_0)\cap \overline{\Om}}\frac{|f(x,u_k)- f(x,u)||u|}{|x|^\alpha} dx \;\mbox{and}\; I_2 = \int_{B_{\frac{r_0}{2}}(x_0)\cap \overline{\Om}}\frac{|f(x,u_k)||u_k-u|}{|x|^\alpha}~ dx.\]

\noi Note that, by H\"{o}lder's inequality and \eqref{e9},
\[I_2= \int_{B_{\frac{r_0}{2}}(x_0)\cap \overline{\Om}}\frac{|f(x,u_k)||u_k-u|}{|x|^\alpha}dx \leq K \left(\int_{\Om}|u_k-u|^{q^{\prime}}dx\right)^{1/{q^{\prime}}}\ra 0\; \mbox{as}\; k\ra \infty.\]
Now, we claim that $I_1\ra 0$. Indeed, given $\e>0$, by density we can take $\phi\in C_{0}^{\infty}(\Om)$ such that $\|u-\phi\|_{q^{\prime}}<\e$. Thus,
{\small \begin{align*}
\int_{B_{\frac{r_0}{2}}(x_0)\cap \overline{\Om}}\frac{|f(x,u_k)- f(x,u)||u|}{|x|^\alpha}dx &\leq \int_{B_{\frac{r_0}{2}}(x_0)\cap \overline{\Om}}(\frac{|f(x,u_k)|}{|x|^\alpha}|u-\phi| + \frac{|f(x,u_k)- f(x,u)|}{|x|^\alpha}|\phi|)dx \\
& \quad\quad +\int_{B_{\frac{r_0}{2}}(x_0)\cap \overline{\Om}}\frac{|f(x,u)|}{|x|^\alpha}|\phi-u|dx.
\end{align*}}
Applying H\"{o}lder's inequality and using \eqref{e9} , we have
\[\int_{B_{\frac{r_0}{2}}(x_0)\cap \overline{\Om}}\frac{|f(x,u_k)|}{|x|^\alpha}|u-\phi| ~dx\leq \left(\int_{B_{\frac{r_0}{2}}(x_0)\cap \overline{\Om}}\left(\frac{|f(x,u_k)|}{|x|^\alpha}\right)^q dx\right)^{1/q}\|u-\phi\|_{q^{\prime}}<\e.\]

\noi Using Lemma \ref{lc1}, we get
\[\int_{B_{\frac{r_0}{2}}(x_0)\cap \overline{\Om}}\frac{|f(x,u_k)- f(x,u)|}{|x|^\alpha}|\phi|~ dx \leq \|\phi\|_{\infty}\int_{B_{\frac{r_0}{2}}(x_0)\cap \overline{\Om}}
\frac{|f(x,u_k)- f(x,u)|}{|x|^\alpha} dx\ra 0.\]
Also from equation \eqref{e9}, we have
\[\int_{B_{\frac{r_0}{2}}(x_0)\cap \overline{\Om}}\frac{|f(x,u)|}{|x|^\alpha}|\phi- u| dx\ra 0\]
and hence the claim. Now to conclude Assertion 1 we use that $K$ is  compact and we repeat the same procedure over a finite covering of balls.

\noi {\bf Assertion 2.} Let $\e_0>0$ be fixed and small enough such that $B_{\e_{0}}(x_i)\cap B_{\e_{0}}(x_j)=\emptyset$ if $i\ne j$ and $\Om_{\e_{0}}= \{x\in \overline{\Om}: \|x-x_j\|\geq \e_0, j=1, 2,\cdots p\}$. Then
\[\int_{\Om_{\e_{0}}}(|\na^{m}u_k|^{\frac{n}{m}-2}\na^m u_k - |\na^{m}u|^{\frac{n}{m}-2}\na^m u)(\na^{m}u_{k}-\na^{m}u)~dx \ra 0.\]
Indeed, let $0<\e<\e_0$ and $\phi\in C^{\infty}_{0}(\mb R^{n}, [0,1])$ be such that $\phi\equiv 1$ in $B_{1/2}(0)$ and
 $\phi=0$ in $\overline{\Om}\setminus B_{1}(0)$. Taking
 \[\psi_{\e}(x)= 1-\sum_{j=1}^{p}\phi\left(\frac{x-x_j}{\e}\right),\]
we have $0\leq \psi_{\e}\leq 1$, $\psi_{\e}\equiv 1$ in $\overline{\Om}_{\e}=\overline{\Om}\setminus \cup_{j=1}^{p}B(x_j,{\e})$, $\psi_{\e}\equiv 0$ in $\cup_{j=1}^{p}B(x_{j},\frac{\e}{2})$ and $\psi_{\e}u_{k}$ is bounded sequence in
$W^{m,\frac{n}{m}}_{0}(\Om)$, for each $\e$. Using \eqref{n7a2} with $v=\psi_{\e}u_{k}$, we have
\[M(\|u_k\|^\frac{n}{m})\int_{\Om} |\na^{m}u_k|^{\frac{n}{m}-2}\na^m u_k \na^{m}(\psi_{\e}u_{k})~dx - \int_{\Om} \frac{f(x,u_k)}{|x|^\alpha}\psi_{\e}u_{k} dx \leq \e_{k}\|\psi_{\e}u_k\|,\]
which implies that
\begin{align}\label{e1}
M(\|u_k\|^\frac{n}{m})\int_{\Om} |\na^{m}u_k|^{\frac{n}{m}}& \psi_{\e}dx+ \sum_{l=1}^{m} \binom{m}{l} \int_{\Om} |\na^{m}u_k|^{\frac{n}{m}-2}\na^m u_k\na^{l}\psi_{\e}\na^{m-l}u_{k}dx\notag\\
&- \int_{\Om} \frac{f(x,u_k)}{|x|^\alpha}\psi_{\e}u_{k} dx \leq \e_{k}\|\psi_{\e}u_k\|.
\end{align}

\noi Now using \eqref{n7a2} with $v=-\psi_{\e}u$, we have
\[M(\|u_k\|^\frac{n}{m})\int_{\Om} |\na^{m}u_k|^{\frac{n}{m}-2}\na^m u_k \na^{m}(- \psi_{\e}u) ~dx + \int_{\Om} \frac{f(x,u_k)}{|x|^\alpha}\psi_{\e}u ~dx \leq \e_{k}\|\psi_{\e}u\|.\]
Hence
\begin{align}\label{e2}
-M(\|u_k\|^\frac{n}{m})\int_{\Om} |\na^{m}u_k|^{\frac{n}{m}-2}&\na^m u_k \na^{m}u \psi_{\e} ~dx - \sum_{l=1}^{m}\binom{m}{l} \int_{\Om} |\na^{m}u_k|^{\frac{n}{m}-2}\na^m u_k\na^{l}\psi_{\e}\na^{m-l}u ~dx\notag\\
& +  \int_{\Om} \frac{f(x,u_k)}{|x|^\alpha}\psi_{\e}u ~dx \leq \e_{k}\|\psi_{\e}u\|.
\end{align}

\noi Using the strict convexity of function $t\mapsto |t|^{n}$, we obtain that
\[0\leq M(\|u_k\|^\frac{n}{m})(|\na^{m}u_k|^{\frac{n}{m}-2}\na^m u_k - |\na^{m}u|^{\frac{n}{m}-2}\na^m u)(\na^{m}u_{k}-\na^{m}u)\]
and consequently
\begin{align*}
0&\leq M(\|u_k\|^\frac{n}{m})\int_{\overline{\Om}_{\e_{0}} }(|\na^{m}u_k|^{\frac{n}{m}-2}\na^m u_k - |\na^{m}u|^{\frac{n}{m}-2}\na^m u)(\na^{m}u_{k}-\na^{m}u)~dx\\
&\leq M(\|u_k\|^\frac{n}{m})\int_{\Om}(|\na^{m}u_k|^{\frac{n}{m}-2}\na^m u_k - |\na^{m}u|^{\frac{n}{m}-2}\na^m u)(\na^{m}u_{k}-\na^{m}u)~dx
\end{align*}

\noi which can be written as
\begin{align}\label{e3}
0 \leq M(\|u_k\|^\frac{n}{m})\int_{\Om}[|\na^{m}u_k|^{\frac{n}{m}}\psi_{\e}-& |\na^{m}u_k|^{\frac{n}{m}-2}\psi_{\e}\na^m u_k\na^{m}u\notag\\
&- |\na^{m}u|^{\frac{n}{m}-2}\psi_{\e}\na^m u\na^{m}u_{k}+ |\na^{m}u|^{\frac{n}{m}}\psi_{\e}]~dx{.}
\end{align}

\noi From \eqref{e1}, \eqref{e2} and \eqref{e3}, we obtain
\begin{align}\label{e5}
0&\leq M(\|u_k\|^\frac{n}{m}) \sum_{l=1}^{m}\binom{m}{l} \int_{\Om} |\na^{m}u_k|^{\frac{n}{m}-2}\na^m u_k\na^{l}\psi_{\e}\na^{m-l}(u -u_k)~dx
+ \int_{\Om} \frac{f(x,u_k)}{|x|^\alpha}\psi_{\e}(u_k-u) ~dx\notag\\
&\quad\quad +M(\|u_k\|^\frac{n}{m})\int_{\Om}|\na^{m}u|^{\frac{n}{m}-2}\na^m u(\na^{m}u- \na^{m}u_{k})\psi_{\e}dx +\e_{k}(\|\psi_{\e}u_k\|+\|\psi_{\e}u\|).
\end{align}
Now we estimate each integral in \eqref{e5}, separately. For arbitrary $\de>0$, using the interpolation inequality
$ab\leq \de a^{\frac{n}{n-m}}+ C_{\de} b^{\frac nm}$ with $C_{\de}= \de^{1-\frac{n}{m}}$ and boundedness of Palais-Smale sequence, we have for all $0\leq r\leq m-1$ and for any $l$,
{\small \begin{align*}
M(\|u_k\|^\frac{n}{m})\int_{\Om} |\na^{m}u_k|^{\frac{n}{m}-2}&\na^m u_k\na^{l}\psi_{\e}(\na^{r}u -\na^{r}u_k)dx\\&\leq C_M\left(\de\int_{\Om} |\na^{m}u_k|^{\frac{n}{m}}dx+ C_{\de}\int_{\Om}|\na^l \psi_{\e}|^{\frac{n}{m}}|\na^{r}u -\na^{r}u_k|^{\frac{n}{m}}dx\right)\\
&\leq C_M\de K + C_MC_{\de} \left(\int_{\Om}|\na^l \psi_{\e}|^{\frac{n}{m}t}dx\right)^{1/t}\left(\int_{\Om}|\na^{r}u -\na^{r}u_k|^{\frac{n}{m}s}dx\right)^{1/s},
\end{align*}}
\noi where $\frac{1}{s}+\frac{1}{t}=1$. Thus since for all $0\leq r\leq m-1$, $\na^{r}u_k \ra \na^{r}u$ strongly in $L^{\frac{ns}{m}}(\Om)$ and
$\de$ is arbitrary, we obtain that
\begin{align}\label{e6}
\limsup_{k\ra\infty} M(\|u_k\|^\frac{n}{m})\int_{\Om} |\na^{m}u_k|^{\frac{n}{m}-2}\na^m u_k\na^{l}\psi_{\e}(\na^{r}u - \na^{r}u_k)~dx \leq 0\; \mbox{for all}\; 0\leq r\leq m-1{.}
\end{align}
Using that $u_k\rightharpoonup u$ weakly in $W^{m,\frac{n}{m}}_{0}(\Om)$, we obtain
\begin{align}\label{e7}
\int_{\Om} |\na^{m}u|^{\frac{n}{m}-2}\na^m u (\na^{m-l}u -\na^{m-l}u_k)\psi_{\e} ~dx\lra 0\;\mbox{as}\;k\ra \infty.
\end{align}
We claim that $\int_{\Om} \frac{f(x,u_k)(u_k-u)}{|x|^\alpha} \psi_{\e} ~dx \ra 0$ as $k\ra \infty$.
Indeed,
\begin{align*}
\int_{\Om} \psi_{\e} \frac{f(x,u_k)(u_k-u)}{|x|^\alpha}dx  = \int_{\Om} \psi_{\e} \frac{f(x,u_k)u_k}{|x|^\alpha} &-  \int_{\Om} \psi_{\e} \frac{f(x,u)u}{|x|^\alpha} dx\\& +\int_{\Om} \psi_{\e} \frac{f(x,u)u}{|x|^\alpha}dx - \int_{\Om} \psi_{\e} \frac{f(x,u_k)u}{|x|^\alpha} dx{.}
\end{align*}
Applying Assertion 1 with $g(x,u)= \psi_{\e}(x) f(x,u)$ and $K=\overline{\Om}_{\e/2}$, we have that
\begin{align}\label{e8}
 \int_{\Om} \psi_{\e} \frac{f(x,u_k)u_k}{|x|^\alpha} ~dx = \int_{\overline{\Om}_{\frac{\e}{2}}} \psi_{\e} \frac{f(x,u_k)u_k}{|x|^\alpha} ~dx\lra \int_{\overline{\Om}_{\frac{\e}{2}}} \psi_{\e} \frac{f(x,u)u}{|x|^\alpha} dx=  \int_{\Om} \psi_{\e} \frac{f(x,u)u}{|x|^\alpha}~ dx,
\end{align}
as $k\ra\infty$. Hence from \eqref{e5}, \eqref{e6}, \eqref{e7} and \eqref{e8}, we have
\begin{equation}\label{usela}
\limsup_{k\ra\infty}M(\|u_k\|^\frac{n}{m})\sum_{l=1}^{m}\binom{m}{l} \int_{\Om} |\na^{m}u_k|^{\frac{n}{m}-2}\na^m u_k \na^l\psi_{\e}(\na^{m-l}u -\na^{m-l}u_k)~dx =0.
\end{equation}

\noi Now, using $\langle J^{\prime}(u_k), \psi_{\e}(u_k- u)\rangle\lra 0$, we conclude the Assertion 2.
Finally using Assertion 2, since $\e_{0}$ is arbitrary we obtain that
\[\na^{m}u_k(x) \ra \na^m u(x)\; \mbox{almost everywhere in }\; \Om,\]
which together with the fact that the sequence $\left\{|\na^m u_k|^{\frac{n}{m}-2}\na^m u_k\right\}$ is bounded in $L^{\frac{n}{n-m}}(\Om)$, implies
\[|\na^m u_k|^{\frac{n}{m}-2}\na^m u_k \rightharpoonup |\na^m u|^{\frac{n}{m}-2}\na^m u\;\mbox{is weakly convergent in}\; L^{\frac{n}{n-m}}(\Om),\]
up to a subsequence. Thus we conclude the proof of Lemma.\QED
\end{proof}


\noi Now we define the Nehari manifold associated to the functional $\mathcal J$, as
\[ \mc{N}:=\{0\not\equiv u\in W_0^{m,\frac{n}{m}}(\Omega):\langle \mathcal J^{\prime}(u),u \rangle=0\}\]
and let $\ds b:=\inf_{u\in \mc{N}} \mathcal J(u)$. Then we need the following Lemma to compare  $c_*$ and $b$.
From the fact that $\frac{f(x,t)}{t^{\frac{2n}{m}-1}}$ increasing we deduce the following
\begin{Lemma}\label{lem7.3}
If condition $(f1)$ holds, then for each $x\in\Omega$, $ sf(x,s)-\frac{2n}{m}F(x,s)$ is increasing for $s\ge0$. In particular $sf(x,s)-\frac{2n}{m}F(x,s)\geq 0$ for all $(x,s)\in \Omega\times [0,\infty)$.
\end{Lemma}

\begin{Lemma}:
If (i) $\frac{m(t)}{t}$ is nonincreasing for $t> 0$ (ii) for each $x\in\Omega, \frac{f(x,t)}{t^{\frac{2n-m}{m}}}$ is increasing for $t>0$ hold.
Then $c_*\leq b $.
\end{Lemma}
\proof Let $u\in \mc{N}$, define $h:(0,+\infty)\rightarrow \mb R$ by $ h(t)=\mathcal J(tu)$. Then using $\langle \mathcal J^\prime(u), u\rangle=0$, we get

 {\small
\begin{align*} h^{\prime}(t)=\|u\|^{\frac{n}{m}} t^{\frac{2n-m}{m}}\left[\left(\frac{M(t^\frac{n}{m}\|u\|^\frac{n}{m})}{t^\frac{n}{m}\|u\|^\frac{n}{m}}-\frac{M(\|u\|^\frac{n}{m})}{\|u\|^\frac{n}{m}}\right)
 +\int_{\Om}\left(\frac{f(u)}{|x|^\alpha (u)^{\frac{2n-m}{m}}}-\frac{f(tu)}{|x|^\alpha(tu)^{\frac{2n-m}{m}}}\right)u^{\frac{2n}{m}}dx\right].\end{align*} }
So $h^{\prime}(1)=0$, $h^{\prime}(t)\geq 0$ for $0<t<1$ and $h^{\prime}(t)<0$ for $t>1$. Hence $\mathcal J(u)=\ds \max_{t\geq0}\mathcal J(tu)$. Now define $g:[0,1]\rightarrow W_0^{m,\frac{n}{m}}(\Omega)$ as  $g(t)=(t_0 u)t$, where $t_0$ is such that $\mathcal J(t_0 u)<0$. We have $g\in \Gamma$ and therefore
\[ c_*\leq\max_{t\in[0,1]}\mathcal J(g(t))\leq \max_{t\geq 0}\mathcal J(tu)=\mathcal J(u).\]
  Since $u\in \mc N$ is arbitrary, $c_*\leq b$ and the proof is complete.\QED

\noi We recall the following singular version of higher integrability Lemma due to Lions \cite{li}. Proof follows from Theorem \ref{admo} and Brezis-Lieb Lemma. For $\alpha=0$, details of the proof can be seen in \cite{sgana}.
\begin{Lemma}\label{pllion}
Let $\{u_k : \|u_k\|=1\}$ be a sequence in $W^{m,\frac{n}{m}}_{0}(\Om)$
converging weakly to a non-zero $u$ and $\na^mu_k\rightarrow \na^mu$ converges pointwise almost every where in $\Omega$.Then for every $p$ such that
$1<p<(1-\frac{\alpha}{n})(1-\|u\|^\frac{n}{m})^{\frac{-m}{n-m}}$,\[\sup_{k}\int_{\Om} \frac{e^{p \beta_{\alpha}|u_k|^{\frac{n}{n-m}}}}{|x|^\alpha}dx< \infty.\]
\end{Lemma}

\noi {\bf Proof of Theorem \ref{main}:} Let $\{u_k\}$ be a Palais-Smale sequence at level $c_*$. That is $\mathcal J(u_k)\rightarrow c_*$ and $\mathcal J^\prime (u_k) \rightarrow 0$. Then by Lemma 2.4 and Lemma 2.7, there exists $u_0\in W^{m,\frac{n}{m}}_{0}(\Om)$ such that $u_k \rightharpoonup u_0$ weakly in $W^{m,\frac{n}{m}}_{0}(\Om)$, $\na^m u_k (x) \rightarrow \na^m u_0(x)$ a.e. in $\Om$.
As $\{u_k\}$ is bounded, so up to a subsequence $\|u_k\|\rightarrow \rho_0>0$. Moreover, condition $\mathcal J^{\prime}(u_k)\rightarrow 0$ and Lemma \ref{IC2} implies that
 \begin{equation}\label{n7a12}
 M(\rho_0^\frac{n}{m})\int_\Omega |\nabla^m u_0|^{\frac{n}{m}-2}\nabla^m u_0\nabla^mv~ dx=\int_\Om \frac{f(x,u_0)}{{|x|^\alpha}}v ~dx\;  \text{for all} \; v\in W_0^{m,\frac{n}{m}}(\Omega).
 \end{equation}
 Now we claim that $u_0$ is the required nontrivial solution.\\
\noi \textbf{Claim 1:} $\ds M(\|u_0\|^\frac{n}{m})\|u_0\|^\frac{n}{m}\geq \int_\Om \frac{f(x,u_0)u_0}{{|x|^\alpha}} dx$.\\
\proof Suppose by contradiction that $M(\|u_0\|^\frac{n}{m})\|u_0\|^\frac{n}{m}< \int_\Om \frac{f(x,u_0)u_0}{|x|^\alpha} ~dx$. That is,\\
 $\langle \mathcal J^{\prime}(u_0),u_0\rangle<0.$ Using \eqref{n1} and Sobolev imbedding, we can see that $\langle \mathcal J^{\prime}(tu_0),u_0\rangle>0$ for t sufficiently small. Thus there exist $\sigma\in(0,1)$ such that $\langle \mathcal J^{\prime}(\sigma u_0),u_0\rangle=0$. That is,  $\sigma u_0\in \mc N$. Thus according to  Lemma \ref{lem7.3},
 \begin{align*}
 c_*&\leq b \leq \mathcal J(\sigma u_0)=\mathcal J(\sigma u_0)-\frac{m}{2n}\langle \mathcal J^{\prime}(\sigma u_0),\sigma u_0\rangle\\
 &=\frac{\widehat M(\|\sigma u_0\|^\frac{n}{m})}{\frac{n}{m}}-\frac{M(\|\sigma u_0\|^\frac{n}{m})\|\sigma u_0\|^\frac{n}{m}}{\frac{2n}{m}}+\frac{m}{2n}\int_\Om \frac{(f(x,\sigma u_0)\sigma u_0-\frac{2n}{m}F(x,
 \sigma u_0))}{|x|^\alpha}\\
 &<\frac{m}{n}\widehat M(\| u_0\|^\frac{n}{m})-\frac{m}{2n}M(\|u_0\|^\frac{n}{m})\| u_0\|^\frac{n}{m}+\frac{m}{2n}\int_\Om \frac{(f(x, u_0)u_0-\frac{2n}{m}F(x,u_0))}{|x|^\alpha}dx
 \end{align*}
 By lower semicontinuity of norm and Fatou's Lemma, we get
 \begin{align*}
 c_*&<  \liminf_{k\rightarrow\infty}\frac{m}{n}\widehat M(\|u_k\|^\frac{n}{m})-\frac{m}{2n}M(\|u_k\|^{\frac{n}{m}})\| u_k\|^\frac{n}{m}\\
 &\quad+\liminf_{k\rightarrow\infty}\frac{m}{2n}\int_\Om \frac{\left(f(x, u_k)u_k-\frac{2n}{m}F(x,u_k)\right)}{|x|^\alpha}dx\\
 &\leq \lim_{k\rightarrow\infty}\left(\mathcal J(u_k)-\frac{m}{2n}\langle \mathcal J^{\prime}(u_k),u_k\rangle\right)=c_*,
 \end{align*}
 which is a contradiction and the claim 2 is proved.\\
\noi \textbf{Claim 2:} $J(u_0)=c_*$.\\
\proof Using $ \int_\Om \frac{F(x,u_k)}{|x|^\alpha}dx\rightarrow \int_\Om \frac{F(x,u_0)}{|x|^\alpha}dx$ and lower semicontinuity of norm we have\\
 $\mathcal J(u_0)\leq c_*$. We are going to show that the case $\mathcal J(u_0)<c_*$ can not occur.\\
 Indeed,  if $\mathcal J(u_0)<c_*$ then $\|u_0\|^\frac{n}{m}<\rho_0^\frac{n}{m}.$ Moreover,
 \begin{equation}\label{7a3}
   \frac{m}{n}\widehat M(\rho_0^\frac{n}{m})=\lim_{k\rightarrow\infty}\frac{m}{n}\widehat M(\|u_k\|^\frac{n}{m})=c_*+\int_\Om \frac{F(x,u_0)}{|x|^\alpha}dx,
 \end{equation}

  which implies  $\ds \rho_0^\frac{n}{m} = \widehat M^{-1}\left(\frac{n}{m}c_*+\frac{n}{m}\int_\Om \frac{F(x,u_0)}{|x|^\alpha}dx\right)$.\\
  \noi Next defining $v_k=\frac{u_k}{\|u_k\|}$ and $v_0=\frac{u_0}{\rho_0}$, we have $v_k\rightharpoonup v_0$ in $W_0^{m,\frac{n}{m}}(\Omega)$ and $\|v_0\|<1$. Thus by Lemma \ref{pllion},
\begin{equation}\label{7a4}
 \sup_{k\in \mb N}\int_\Omega \frac{e^{p|v_k|^{\frac{n}{n-m}}}}{|x|^\alpha}~dx<\infty\;\;  \text{for all}\;  1<p<\left(1-\frac{\alpha}{n}\right)\beta_{\alpha}{\left(1-\|v_0\|^\frac{n}{m}\right)^\frac{-m}{n-m}}.
 \end{equation}
\noi On the other hand, by Assertion 2, \eqref{7a0} and Lemma \ref{lem7.3}, we have
\[  \mathcal J(u_0)\ge \frac{m}{n}\widehat M(\|u_0\|^\frac{n}{m})-\frac{2m}{n}M(\|u_0\|^\frac{n}{m})\|u_0\|^\frac{n}{m}+\frac{m}{2n}\int_\Om \frac{( f(x,u_0)u_0-\frac{2n}{m} F(x,u_0)}{|x|^\alpha}.\]
So, $\mathcal J(u_0)\geq 0$. Using this together with Lemma \ref{le1} and the equality,
$\frac{n}{m}(c_*-J(u_0))=\widehat M(\rho_{0}^\frac{n}{m})-\widehat M(\|u_0\|^\frac{n}{m})$
we get
$\widehat M(\rho_{0}^\frac{n}{m}) \le \frac{n}{m} c_*+\widehat M(\|u_0\|^\frac{n}{m})< \widehat M(\beta_{\alpha}^\frac{n-m}{m})+\widehat M(\|u_0\|^\frac{n}{m})$
and therefore by $(m1)$
\begin{equation}\label{7a6}
\rho_{0}^\frac{n}{m}< \widehat M^{-1}\left(\widehat M(\beta_{\alpha}^\frac{n-m}{m})+\widehat M(\|u_0\|^\frac{n}{m})\right)\le \beta_{\alpha}^\frac{n-m}{m}+\|u_0\|^\frac{n}{m}.
\end{equation}
Since $\rho_{0}^\frac{n}{m}(1-\|v_0\|^\frac{n}{m})=\rho_{0}^\frac{n}{m}-\|u_0\|^\frac{n}{m}$, from \eqref{7a6} it follows that
\[ \rho_{0}^\frac{n}{m}< \frac{\beta_{\alpha}^\frac{n-m}{m}}{1-\|v_0\|^\frac{n}{m}}.\]
Thus, there exists $\nu>0$ such that $ \|u_k\|^{\frac{n}{n-m}} < \nu < \frac{\beta_{\alpha}}{(1-\|v_0\|^\frac{n}{m})^{\frac{m}{n-m}}}$ for $k$ large. We can choose $q>1$ close to $1$ such that $q \|u_k\|^{\frac{n}{n-m}} \le  \nu < \frac{\beta_{\alpha}}{(1-\|v_0\|^\frac{n}{m})^{\frac{m}{n-m}}}$ and using \eqref{7a4}, we conclude that for $k$ large
\[ \int_\Om \frac{e^{q |u_k|^{\frac{n}{n-m}}}}{|x|^\alpha} dx \le \int_\Om \frac{e^{\nu |v_k|^{\frac{n}{n-m}}}}{|x|^\alpha}dx \le C.\]
Now by standard calculations, using H\"{o}lder's inequality and weak convergence of  $\{u_k\}$ to $u_0$, we get $\int_\Om \frac{f(x,u_k) (u_k-u_0)}{|x|^\alpha}dx \rightarrow 0$ as $k\rightarrow \infty$. Since $\langle \mathcal J^\prime (u_k), u_k-u_0\rangle \rightarrow 0$, it follows that
\begin{equation} \label{na7new}
M(\|u_k\|^\frac{n}{m}) \int_{\Om} |\na^m u_k|^{\frac{n}{m}-2} \na^m u_k (\na^m u_k-\na^m u_0) \rightarrow 0.
\end{equation}
\noi On the other hand, using $u_k\rightharpoonup u_0$ weakly and boundedness of $\{u_k\}$, we get
\begin{equation}\label{na8new}
M(\|u_k\|^\frac{n}{m}) \int_{\Om} |\na^m u_0|^{\frac{n}{m}-2} \na^m u_0 (\na^m u_k - \na^m u_0)\ra 0\; \mbox{as}\; k\ra\infty.
\end{equation}
Subtracting \eqref{na8new} from \eqref{na7new}, we get
\[M(\|u_k\|^\frac{n}{m}) \int_\Om (|\na^m u_k|^{\frac{n}{m}-2}\na^m u_k - |\na^m u_0|^{\frac{n}{m}-2} \na^m u_0)\cdot (\na^m u_k -\na^m u_0) \rightarrow 0 \]
as $k\rightarrow \infty$. Now using this and inequality by using
\begin{align*}
|a-b|^l\leq 2^{l-2}(|a|^{l-2}a-|b|^{l-2}b)(a-b)\;\mbox{for all}\; a,b\in \mb R^n\;\mbox{and}\; l\geq 2,
\end{align*}
 with $a=\na^m u_k$ and $b=\na^m u_0$, we get
\[M(\|u_k\|^\frac{n}{m})\int_{\Om}|\na^m u_k- \na^m u_0|^\frac{n}{m} \ra 0\;\mbox{as}\; k\ra\infty.\]
Since $M(t)\ge M_0$, we obtain $u_k\ra u$ strongly in $W^{m,\frac{n}{m}}_{0}(\Om)$ and hence $\|u_k\| \rightarrow \|u_0\|$.\\
Therefore, $J(u_0)=c_*$ and hence the claim.\\
\textbf{Claim 3:} $u_0\not\equiv 0$. Suppose not. Then from claim 2, \\
$\ds\lim_{k\ra \infty}\widehat M(\|u_k\|^{\frac{n}{m}})= \frac{n}{m}\left(c+ \int_{\Om} \frac{F(x,u)}{|x|^\alpha}dx\right)= \frac{nc}{m}$ and $c\ne 0$ which is a contradiction. Hence $u\not\equiv 0$. Now By Assertion 3 and \eqref{7a3} we can see that $\widehat M(\rho_{0}^\frac{n}{m})=\widehat M(\|u_0\|^\frac{n}{m})$ which shows that $\rho_{0}^\frac{n}{m}=\|u_0\|^\frac{n}{m}$. Hence by \eqref{n7a12} we have
\[  M(\|u_0\|^\frac{n}{m}) \int_\Om |\na^m u_0|^{\frac{n}{m}-2}\na^m u_0 \na^m v~ dx =\int_\Om \frac{f(x,u_0)}{|x|^\alpha} v ~dx, \; \text{for all} \; v\in W^{m,\frac{n}{m}}_{0}(\Om).\] Thus, $u_0$ is a solution of $(\mc P)$. \QED
\section{Multiplicity results for convex-concave problem}
\setcounter{equation}{0} Let $\Om \subset\mb R^n$, $n\geq 2m\geq 2$,
be a bounded domain with smooth boundary. We consider the following
quasilinear polyharmonic equation in $\Om$:
$$ (\mathcal P_{\la})\quad \left\{
\begin{array}{lr}
 \quad  {-M\left(\displaystyle \int_\Omega |\nabla^mu|^\frac{n}{m}dx\right)}\De^{m}_{\frac{n}{m}} u = \la h(x)|u|^{q-1}u+ u|u|^{p} ~ \frac{e^{|u|^{\ba}}}{|x|^\alpha} \; \text{in}\;
\Om{,}\\
 \quad \quad u=\nabla u=\cdot\cdot\cdot= {\nabla}^{m-1} u =0\quad\quad \text{on} \quad \partial \Om{,}
 \end{array}
\right.
$$
 where $1< q<\frac{n}{m}<\frac{2n}{m}< p+2$, $0<\alpha<n$, $\ba\in (1,\frac{n}{n-m}]$ and $\la>0$.
We assume the following:
 \begin{enumerate}
\item[$(m4)$] $M(t)=a+bt$ with $M(t)\geq M_0$ for all $t\geq0{.}$
\item[$(H)$]$h^{\pm}\not\equiv 0$ and $\ds h\in L^{\frac{k}{k-1}}(\Om)\cap L^{\ga}(\Om, |x|^\frac{\alpha}{k-1}~dx)$, where $\gamma=\frac{n}{n-mq-m}$ and $k=\frac{p+2+\beta}{q}$.

\end{enumerate}

\begin{Definition}
We say that $u\in W^{m,\frac{n}{m}}_{0}(\Om)$ is a weak solution of
$(\mathcal P_{\la})$, if for all $\phi \in W^{m,\frac{n}{m}}_{0}(\Om)$, we have
\begin{align*}
M\left(\displaystyle \int_\Omega |\nabla^m u|^\frac{n}{m}dx\right)\int_{\Om}|{\na}^{m} u|^{\frac{n}{m}-2}{\na}^{m} u {\na}^{m} \phi ~dx = \la \int_{\Om}
h(x)|u|^{q-2}u\phi ~dx+\int_{\Om}\frac{f(u)}{|x|^\alpha}\phi ~dx.
\end{align*}
\end{Definition}

\noi The Euler functional associated with the problem $(\mc P_{\la})$ is $J_{\la} :W^{m,\frac{n}{m}}_{0}(\Om) \lra \mb{R}$ defined as
\begin{equation*}
J_{\la}(u)= \frac{m}{n}\widehat M\left(\int_{\Om}|\nabla^{m}u|^{\frac{n}{m}} ~dx\right) - \frac{\la}{q}\int_{\Om} h(x)|u|^{q} ~dx- \int_{\Om}
\frac{F(u)}{|x|^\alpha} ~dx,
\end{equation*}
where $f(u)=u|u|^{p}e^{|u|^{\ba}}$ and $\ds F(u)=\int_{0}^{u} f(s)
ds$.


\noi {In this part {of the paper}, we show the following multiplicity result:
\begin{Theorem}\label{tp2}
For $\ba\in (1,\frac{n}{n-m})$, $h$ satisfying $(H)$ and $\la\in(0,\la_0)$, $(\mc P_{\la})$ admits atleast two solutions.
\end{Theorem}
\noi Moreover, we have the following existence result for critical case:
\begin{Theorem}\label{tp1}
 Let $\ba=\frac{n}{n-m}$ and $h$ satisfies $(H)$. Then there exists $\la^0$ such that for $\la\in(0,\la^0)$, $(\mc P_{\la})$ admits a solution $u_{\la}$,
which is a local minimizer of $J_{\la}$ on $W_{0}^{m,\frac{n}{m}}(\Om)$.
\end{Theorem}

}
\subsection{The Nehari manifold and Fibering map analysis for $(P_{\la})$}
\setcounter{equation}{0} \noi The energy functional $J_{\la}$ is
unbounded on the space $W^{m,\frac{n}{m}}_{0}(\Om)$, but is
bounded below on an appropriate subset $\mc N_\la$ (see below) of
$W^{m,\frac{n}{m}}_{0}(\Om)$ and a minimizer on nonempty decompositions of this set
give rise to the solutions of $(\mathcal P_{\la})$. We introduce the Nehari
manifold $\mc N_\la$ associated to the problem $(\mathcal P_\la)$ as
\begin{equation*}
\mc N_{\la}: = \left\{u\in W^{m,\frac{n}{m}}_{0}(\Om): \ld
J_{\la}^{\prime}(u),u\rd=0 \right\},
\end{equation*}
where $\ld\;,\; \rd$ denotes the duality between
$W^{m,\frac{n}{m}}_{0}(\Om)$ and its dual space. We note that
$\mathcal N_{\la}$ contains every solution of $(\mc P_{\la})$. The
Nehari manifold is closely related to the behavior of the fibering
map $\phi_u: \mb R^+\ra \mb R$  defined as
$\phi_{u}(t)=J_{\la}(tu)$, which was introduced by Drabek and
Pohozaev in \cite{DP}. For $u\in W^{m,\frac{n}{m}}_{0}(\Om)$, we
have
\begin{align*}
\phi_{u}(t) &= \frac{m}{n}\widehat M\left(t^\frac{n}{m} \|u\|^{\frac{n}{m}}\right) - \frac{\la
 t^{q}}{q} \int_{\Om} h(x)|u|^{q} ~dx -\int_{\Om} \frac{F(tu)}{|x|^\alpha} ~dx ,\\
\phi_{u}^{\prime}(t) &= M\left(t^\frac{n}{m} \|u\|^{\frac{n}{m}}\right)t^{\frac{n}{m}-1}\|u\|^{\frac{n}{m}} - {\la t^{q-1}} \int_{\Om} h(x)|u|^{q} ~dx - \int_{\Om}\frac{{f(tu)u}}{|x|^\alpha} ~dx,\\
\phi_{u}^{\prime\prime}(t) &= \left(\frac{n}{m}-1\right) t^{\frac{n}{m}-2} M\left(t^\frac{n}{m}\|u\|^{\frac{n}{m}}\right)\|u\|^\frac{n}{m}+\frac{n}{m}M^\prime \left(t^\frac{n}{m}\|u\|^{\frac{n}{m}}\right)t^{\frac{2n}{m}-1}\|u\|^{\frac{2n}{m}}\\
 &- \la (q-1)t^{q-2} \int_{\Om} h(x)|u|^{q} ~dx - \int_{\Om}\frac{{f^{\prime}(tu)u^2}}{|x|^\alpha}~ dx.
\end{align*}
So, $u\in \mathcal N_{\la}$ if and only if
\begin{equation*}
M\left( \|u\|^{\frac{n}{m}}\right)\|u\|^{\frac{n}{m}} - {\la} \int_{\Om} h(x)|u|^{q} ~dx - \int_{\Om}\frac{{f(u)u}}{|x|^\alpha} ~dx=0.
\end{equation*}
It is easy to see that $tu\in \mathcal N_{\la}$ if and only if $\phi_u^\prime(1)=0$, that is
$\phi_{u}^{\prime}(t)=0$. We divide $\mathcal N_{\la}$ into three parts corresponding to local minima,
local maxima and points of inflection of $\phi_u(t)$ as,
\[
\mathcal N_{\la}^{\pm}= \left\{u\in \mc N_{\la}:
\phi_{u}^{\prime\prime}(1)
\gtrless0\right\} \quad
\mathcal N_{\la}^{0}= \left\{u\in \mc N_{\la}:
\phi_{u}^{\prime\prime}(1) = 0\right\}.
\]

To prove our result, first we prove the following important Lemma:
\begin{Lemma}\label{le3}
Let $\ds \La := \left\{u\in W^{m,\frac{n}{m}}_{0}(\Om),\;\;
 {\|u\|^{\frac{3n}{2m}} \leq \frac{m}{2\sqrt {ab(n-mq)(2n-mq)}}}\int_{\Om}\frac{{f^{\prime}(u){u}^2}}{|x|^\alpha}dx\right\}$. \\
Then there exists $\la_0>0$ such that
for every $\la\in(0,\la_0)$,
\begin{equation}\label{eq8}
I_\la:=\inf_{u\in \La\setminus\{0\}}\left\{\int_{\Om}
\frac{\left(mp+2m-2n +m\ba{|u|^{\ba}}\right)|u|^{p+2} e^{|u|^{\ba}}}{|x|^\alpha}- (2n-mq)
{\la} \int_{\Om} h|u|^{q}dx \right\}>0.
\end{equation}
\end{Lemma}
\proof ~
 Step 1: We have \[\ds \inf_{u\in \La\setminus\{0\}} \|u\|
>0.\]

\noi Suppose not. Then there exists a sequence $\{u_k\}
\subset \La\setminus\{0\}$ such that $\|u_k\| \ra 0$
and we have
\begin{align}\label{eq9}
{ \|u_k\|^{\frac{3n}{2m}} \leq  \frac{m}{2\sqrt{ab(n-mq)(2n-mq)}}}\int_{\Om}\frac{{f^{\prime}(u_k){u^2_k}}}{|x|^\alpha}dx
\;\;\; \mbox{for all} \;\;k .
\end{align}
From $f(u)=u|u|^pe^{|u|^{\ba}}$, by H\"{o}lder's and Sobolev
inequalities, we have
\begin{align*}
\int_{\Om} \frac{f^{\prime}(u_k)u_{k}^2}{|x|^\alpha} dx &=
\int_{\Om} \frac{\left(p+1+\ba|u_k|^{\ba}\right)|{u_k}|^{p+2}e^{|u_k|^{\ba}}}{|x|^\alpha} dx\\
&\leq C\int_{\Om}\frac{|{u_k}|^{p+2} e^{(1+\de)|u_k|^{\ba}}}{|x|^\alpha} dx\\
 &\leq C\left(\int_{\Om}\frac{|{u_k}|^{(p+2)t^{\prime}}}{|x|^\alpha} dx
\right)^{\frac{1}{t^{\prime}}}\left( \int_{\Om} \frac{e^{t(1+\de)|u_k|^{\ba}}}{|x|^\alpha}  dx\right)^{\frac{1}{t}}\\
&\leq C' \|u_k\|^{p+2} \left(\sup_{\|w_k\|\leq 1}
\int_{\Om}\frac{e^{t(1+\de)\|u_k\|^{\ba}|w_k|^{\ba}}}{|x|^\alpha}
dx\right)^{\frac{1}{t}}{.}
\end{align*}
Since $\|u_k\|\ra 0$ as $k\ra\infty$, we can choose $\gamma=t(1+\de)\|u_k\|^{\ba}$ such that $\gamma \leq \ba_{\alpha}$.
Hence for every $\ba\leq\frac{n}{n-m}$, using Theorem \ref{admo} and \eqref{eq9}, we get $1\leq
K'\|u_k\|^{p+2-\frac{n}{m}}\ra 0$ as $k\ra \infty$, since $p+2 >\frac{n}{m}$.
This is a contradiction.\\
\noi Step 2: Let $\ds C_1=\inf_{u\in \La\setminus\{0\}}\left\{\int_{\Om}
\frac{\left(mp+2m-2n +m\ba{|u|^{\ba}}\right)|u|^{p+2} e^{|u|^{\ba}}}{|x|^\alpha}dx\right\}$. Then $C_{1}>0$.

\noi From Step 1 and the definition of $\La$, we obtain
\begin{align*}
0&< \inf_{u\in \La\setminus\{0\}}\int_{\Om}
\frac{f^{\prime}(u)u^2}{|x|^\alpha}dx =\inf_{u\in \La\setminus\{0\}}\int_{\Om}\frac{{\left(p+1+\ba{|u|^{\ba}}\right)|u|^{p+2}
e^{|u|^{\ba}}}}{|x|^\alpha} dx.
\end{align*}
Using this it is easy to check that
\begin{align*}
 \inf_{u\in \La\setminus\{0\}}\left\{\int_{\Om}
\frac{\left(mp+2m-2n +m\ba{|u|^{\ba}}\right)|u|^{p+2} e^{|u|^{\ba}}}{|x|^\alpha}dx\right\}>0.
\end{align*}

\noi Step 3: Let
$\la<\left(\frac{1}{2n-mq}\right)(\frac{C_1}{l})^{\frac{(k-1)}{k}}$, where
$l=\int_{\Om}|h(x)|^{\frac{k}{k-1}}|x|^\frac{\alpha}{k-1}dx$ and $k=\frac{p+2+\ba}{q}$. Then (\ref{eq8}) holds.

\noi Using H\"{o}lder's inequality and ${(H)}$ we have,
\begin{align*}
\int_{\Om} h(x)|u|^{q}
&\leq  \left(\int_{\Om}|h(x)|^{\frac{k}{k-1}}|x|^\frac{\alpha}{k-1}dx\right)^{\frac{k-1}{k}} \left(\int_{\Om} \frac{|u|^{qk}}{|x|^\alpha} dx \right)^{\frac{1}{k}}\\
&= l^{\frac{k-1}{k}} \left(\int_{\Om} \frac{|u|^{p+2+\ba}}{|x|^\alpha} dx \right)^{\frac{1}{k}}\\
&\leq l^{\frac{k-1}{k}}
\left(\int_{\Om}\frac{\left(mp+2m-2n +m\ba{|u|^{\ba}}\right)
|u|^{p+2}e^{|u|^{\ba}}}{|x|^\alpha} dx \right)^{\frac{1}{k}}\\
&\leq\left(\frac{l}{C_1}\right)^{\frac{k-1}{k}}
\int_{\Om}\frac{\left(mp+2m-2n +m\ba{|u|^{\ba}}\right) |u|^{p+2}e^{|u|^{\ba}}}{|x|^\alpha}
dx .
\end{align*}
The above inequality combined with step 2 proves the Lemma.\QED

\begin{Remark} From the proof of Lemma \ref{le3}, it is clear that if $I_{\la^\prime}>0$ for some $\la^\prime$, then $I_\la >0$ for all $\la <\la^\prime$. So we can assume that $\la_0$ is the maximum of all $\la$ such that $I_\la>0$ for all $\la\in (0,\la_0)$.
\end{Remark}
\noi Now we discuss the behavior of fibering map with respect to sign of $\int_\Omega h(x)|u|^qdx$ \\
\textbf{Case 1:} $\int_\Omega h(x)|u|^qdx>0$\\
 We define $\psi_{u}: \mb R^{+} \lra
\mb R$ by
\begin{equation*}
\psi_{u}(t)= t^{\frac{n}{m}-q} M\left(t^\frac{n}{m}\|u\|^{\frac{n}{m}}\right) \|u\|^{\frac{n}{m}}-  t^{1-q}\int_{\Om}\frac{f(tu) u}{|x|^\alpha} ~dx.
\end{equation*}
Clearly, for $t>0$, $t$ is a
solution of $\psi_{u}(t)={\la} \int_{\Om} h(x)|u|^{q}~dx$ if and only if $tu\in \mc N_{\la}$. Now
\begin{align*}
\psi_{u}^{\prime}&(t) = \left(\frac{n}{m}-q\right) t^{\frac{n}{m}-1-q} \|u\|^{\frac{n}{m}}M\left(t^\frac{n}{m}\|u\|^{\frac{n}{m}}\right)\\
&+\frac{n}{m}t^{\frac{2n}{m}-1}\|u\|^{\frac{2n}{m}}M^\prime\left(t^\frac{n}{m}\|u\|^{\frac{n}{m}}\right) - t^{1-q}
\int_{\Om}\frac{f^{\prime}(tu)u^2}{|x|^\alpha} dx-(1-q)t^{-q} \int_{\Om}\frac{f(tu)u}{|x|^\alpha} dx \\
&=\left(\frac{n}{m}-q\right) t^{\frac{n}{m}-1-q} a\|u\|^{\frac{n}{m}}+\left(\frac{2n}{m}-q\right) t^{\frac{2n}{m}-1-q} b\|u\|^{\frac{2n}{m}}\\
& - t^{1-q}\int_{\Om}\frac{f^{\prime}(tu)u^2}{|x|^\alpha} dx-(1-q)t^{-q} \int_{\Om}\frac{f(tu)u}{|x|^\alpha} dx{.}
\end{align*}
{
Observe that $\ds \lim_{t\ra 0^{+}}\psi_{u}^{\prime}(t)=\infty$, $\ds \lim_{t\ra \infty}\psi_{u}^{\prime}(t)=-\infty$  and $\ds \lim_{t\ra \infty}\psi_{u}(t)=-\infty$. So it is easy easy to show that there exists $t_*=t_{*}(u)>0$ such that $\psi_{u}^{\prime}(t_*)=0$ and $\max \psi_u(t)=\psi_u(t_*)>0$ . Also for any critical point of $\psi_u$  we have
\begin{equation}\label{3.17}
\left(\frac{n}{m}-q\right)t^\frac{n}{m}a\|u\|^\frac{n}{m}+\left(\frac{2n}{m}-q\right)t^\frac{2n}{m}b\|u\|^\frac{2n}{m}-(1-q)\int_\Omega\frac{f(tu)tu}{|x|^\alpha}dx
-\int_\Omega\frac{f'(tu)(tu)^2}{|x|^\alpha}dx=0
\end{equation}
Now using \eqref{3.17}, we get $\psi_u(t)>0$ for any critical point $t$  of $\psi_u$. Therefore for small values of $\lambda>0, \;t_*{u}\in \La\setminus\{0\}$ and }
\begin{align*}
\psi_{u}(t_*){\geq} \frac{1}{t_{*}^{q}(2n-mq)}\left[ m\int_{\Om}\frac{f^{\prime}(t_{*}u)
(t_*u)^2}{|x|^\alpha}  dx -(2n-m)\int_{\Om}\frac{f(t_*u) t_*u}{|x|^\alpha}  dx \right].
\end{align*}
Note that
\begin{equation}\label{interm}
mf^\prime(s)s^2-(2n-m)f(s)s=(mp+2m-2n+m\ba|s|^{\ba})|s|^{p+2}
e^{|s|^{\ba}}.
\end{equation}
Using Lemma \ref{le3}, we obtain
\begin{align*}
\psi_{u}(t_*)&- {\la} \int_{\Om} h(x)|u|^{q}dx\\& {\geq}\frac{1}{ t_{*}^{q} (2n-mq)}\left(
\int_{\Om}\frac{( mp+2m-2n+m\beta u^\beta )u^{p+2}e^{u^\beta}}{|x|^\alpha}dx-(2n-m)\lambda\int_\Omega h(x)|u|^qdx  \right)\\&>0.
\end{align*}

\noi Then there exist  unique
$t_{1}=t_{1}(u)<t_*$, $t_{2}=t_{2}(u)>t_*$ such that
$\psi_{u}(t_{1})= {\la} \int_{\Om} h(x)|u|^{q}dx = \psi_{u}(t_{2})$ and $t_{1}{u}$,
$t_{2}{u}\in \mc N_{\la}$. Also $\psi_{u}^{\prime}(t_{1})>0$,
$\psi_{u}^\prime(t_{2})<0$ implies $t_{1}u\in \mc N_{\la}^{+}$ and $t_{2}u\in
\mc N_{\la}^{-}$. As $\phi_{u}^{\prime}(t)=t^{q-1}(\psi_{u}(t)- \la
\int_{\Om} h(x)|u|^{q}dx)$, we have $\phi_{u}^{\prime}(t)<0$ for all $t\in [0,t_1)$ and
$\phi_{u}^\prime(t)>0$ for all $t\in (t_1,t_2)$. Thus
$\ds\phi_{u}(t_1)=\min_{0\leq t\leq t_2} \phi_{u}(t)$. Also
$\phi_{u}^{\prime}(t)>0$ for all $t\in [t_*,t_2)$,
$\phi_{u}^\prime(t_2)=0$ and $\phi_{u}^{\prime}(t)<0$ for all $t\in
(t_2,\infty)$ implies that $\ds \phi_{u}(t_2)= \max_{t\geq t_*}
\phi_{u}(t)$. \\

\noi {\bf Case 2:} $\int_\Omega h(x)|u|^qdx{\leq}0$.\\
\noi Since $\int_{\Om} h |u|^{q} ~dx \leq 0,$ and $\psi_u(t)\rightarrow -\infty$ as $t\rightarrow \infty$. {Since for any critical point $t$ of $\psi_u$, $\psi_u(t)>0$. Therefore for any $\lambda>0$ it} is clear that there exists $t_2(u)$ such that
$\psi_{u}(t_2)={\la} \int_{\Om} h(x)|u|^{q}.$ Thus for $0<t<t_2$,
$\phi_{u}^{\prime}(t)>0$ and for $t>t_2$, $\phi_{u}^{\prime}(t)<0.$
Thus $\phi_{u}$ has exactly one critical point $t_{2}(u)$, which is
a global maximum point. Hence $t_{2}(u)u \in \mc N^{-}_{\la}$.\\

\noi Now from above discussion we have the following Lemma
\begin{Lemma}\label{le5}
For any $u\in W_0^{m,\frac{n}{m}}(\Omega)$,
\begin{enumerate}
\item There exists a unique $t_2(u)>0$ such that $t_2u \in \mc N_\lambda^-$ for every $\lambda>0$.
\item {Assume that $\lambda$ satisfies conditions in Lemma \ref{le3}. If $\int_\Omega h(x)|u|^qdx >0$, set $t_*(u)$ the minimal positive critical point of $\psi_w$ such that $\displaystyle\max_{t\in\mathbb R^+}\psi_u(t)=\psi_u(t_*)$. Then, there exists unique $t_1(u)<t_*(u)<t_2(u)$ such that $t_1 u\in \mc N^+_\lambda$ and $t_2 u\in \mc N^-_\lambda$. Moreover $J_\lambda(t_1 u)<J_\lambda(tu)$ for any $t\in [0,t_2]$ such that $t\neq t_1$ and $J_\lambda(t_2u)=\displaystyle \max_{t>t_*}J_\lambda(tu)$.}
\end{enumerate}
\end{Lemma}

 \noi We note that for $u\in W^{m,\frac{n}{m}}_{0}(\Om)$, $u\mapsto t_2(u)$ is a continuous function for non-zero $u$. This follows from the uniqueness of $t_2(u)$ and the extremal property of $t_2(u)$.
\begin{Lemma}
If $\la$ be such that
(\ref{eq8}) holds. Then $\mc N_{\la}^{0}= \{0\}$.
\end{Lemma}

\proof~ Suppose $u\in \mc N_{\la}^{0}$, $u\not\equiv 0$. Then by
definition of $\mc N_{\la}^{0}$, we have the following two equations
\begin{align}\label{eq10}
\left(\frac{n}{m}-1\right)a\|u\|^{\frac{n}{m}}+\left(\frac{2n}{m}-1\right)b\|u\|^\frac{2n}{m}&=\la (q-1) \int_{\Om} h(x)|u|^{q}~ dx+\int_{\Om}\frac{f^{\prime}(u)u^2}{|x|^\alpha}~ dx,
\end{align}
\begin{align}\label{eq11}
a\|u\|^{\frac{n}{m}}+b\|u\|^\frac{2n}{m}&=\la \int_{\Om} h(x)|u|^{q} ~dx+\int_{\Om}\frac{f(u)u}{|x|^\alpha} ~dx.
\end{align}
Then from above equations, using inequality between arithmetic and geometric mean, we can easily deduce that $u\in\La\setminus\{0\}$. Now from \eqref{interm}, \eqref{eq10} and \eqref{eq11} we get
\begin{align*}
(2n-m)\la \int_{\Om} h(x)|u|^{q} ~dx{>}
\int_{\Om}\frac{\left(mp+2m-2n+m\ba|u|^{\ba}\right)|u|^{p+2}e^{|u|^{\ba}}}{{|x|^\alpha}} ~dx,
\end{align*}
which violates Lemma \ref{le3}. Hence $\mc N_{\la}^{0}= \{0\}$.\QED

\noi The following lemma shows that minimizers for $J_{\la}$ on any of
subsets $\mc N_{\la}^{+}, \mc N_{\la}^{-}$ of $\mc N_{\la}$ are usually critical points for $J_{\la}$. Proof is standard as can be seen in Lemma 3.8   of \cite{Racsam}.
\begin{Lemma}
If $u$ is a minimizer of ${J}_{\lambda}$ on $\mathcal{N}_{\lambda}$ such that $u \notin \mathcal{N}_{\lambda}^{0}.$ Then $u$ is a critical point for ${J}_{\lambda}.$
\end{Lemma}

\noi We define $\ds\theta_{\la} := \inf \left\{J_{\la}(u)\mid u\in
\mc N_{\la}\right\}$ and prove the following lower bound:
\begin{Lemma}\label{th1}
$J_{\la}$ is coercive and bounded below on $\mc N_{\la}$. Moreover, there exists
a constant $C>0$ depending on $p$, $q$, $n$ and $m$ such that $\theta_{\la} \geq - C
\la^{\frac{k}{k-1}}$, where $ k=\frac{p+2+\ba}{q}.$
\end{Lemma}
\proof ~ Let $u \in \mc N_{\la}$. Then we have
\begin{align}\label{eq4}
 J_{\la}(u)
=\left(\frac{m}{n}-\frac{1}{p+2}\right)a\|u\|^{\frac{n}{m}}+\left(\frac{m}{2n}-\frac{1}{p+2}\right)b\|u\|^\frac{2n}{m} &+\int_{\Om}\frac{\frac{1}{p+2}f(u)u-F(u)}{|x|^\alpha}\notag\\&-\lambda\left(\frac{1}{q}-\frac{1}{p+2}\right)\int_{\Om}h|u|^{q}dx.
\end{align}
Using $F(s)\leq \frac{1}{p+2}f(s)s$ for all $ s\in \mb R$, H\"{o}lder's  and
Sobolev inequalities in \eqref{eq4}, we obtain
\begin{align*}
 J_{\la}(u)
&\geq \left(\frac{m}{n}-\frac{1}{p+2}\right)a\|u\|^{\frac{n}{m}}+\left(\frac{m}{2n}-\frac{1}{p+2}\right)b\|u\|^\frac{2n}{m}
-\lambda\left(\frac{1}{q}-\frac{1}{p+2}\right)C_0\|u\|^q
\end{align*}
for some constant $C_0>0$, which shows $J_{\la}$ is coercive on
$\mc N_{\la}$ as $q<\frac{n}{m}$.

\noi Again for $u \in \mc N_{\la}$ we also have
\begin{align*}
J_{\la}(u) &{\geq} \frac{m}{2n} \int_{\Om}\frac{f(u)u}{|x|^\alpha} ~dx- \int_{\Om} \frac{F(u)}{|x|^\alpha}~ dx
-{\la}\left(\frac{1}{q}-\frac{m}{2n}\right) \int_{\Om} h(x)|u|^{q}~dx.
\end{align*}
If $\int_\Omega h(x)|u|^q~dx<0$, then $J_\la(u)$ is bounded below by $0$. If
$\int_\Omega h(x)|u|^qdx>0$ then by using H\"{o}lder's inequality, we have
\begin{equation*}
\int_{\Om} h(x)|u|^{q} ~dx \leq  l^{\frac{k-1}{k}} \left(\int_{\Om} \frac{|u|^{qk}}{|x|^\alpha}~
dx\right)^{\frac{1}{k}},
\end{equation*}
where $l=\int_{\Om} |h(x)|^{k/k-1}|x|^\frac{\alpha}{k-1}~ dx.$ Also, It is easy to see that
\begin{equation}\label{reqla}
\frac{m}{2n}f(u)u- F(u)\geq
\left(\frac{m}{2n}-\frac{1}{p+2}\right)|u|^{p+2+\ba}.
\end{equation}
From these inequalities, we obtain
\begin{align}
J_{\la}(u)&\geq \left(\frac{mp+2m-2n}{2n(p+2)}\right)\int_{\Om}
\frac{|u|^{qk}}{|x|^\alpha} ~dx -
\frac{\la(2n-mq)l^{\frac{k-1}{k}}}{2nq}\left(\int_{\Om}
\frac{|u|^{qk}}{|x|^\alpha} ~dx\right)^{\frac{1}{k}}.\nonumber
\end{align}
 Consider the function $\rho(x): \mb R^+
\lra \mb R$ defines as
$$\rho(x)=\left(\frac{mp+2m-2n}{2n(p+2)}\right)x^k -
\frac{\la(2n-mq)l^{\frac{k-1}{k}}}{2nq} x.$$
Then $\rho(x)$ attains its minimum at $\left(\frac{\la(2n-mq)(p+2)l^{\frac{k-1}{k}}}{kq(mp+2m-2n)}\right)^{\frac{1}{k-1}}$. So,
\begin{equation*}
\inf_{u\in \mc N_{\la}}J_{\la}(u) \geq
\rho\left(\frac{\la(2n-mq)(p+2)l^{\frac{k-1}{k}}}{kq(mp+2m-2n)}\right)^{\frac{1}{k-1}}.
\end{equation*}
From this we obtain that
\begin{equation*}
\theta_{\la} \geq -C(p,q,n,m)\la^{\frac{k}{k-1}},
\end{equation*}
where $C(p,q,n,m)=
\left(\frac{1}{k^{\frac{1}{k-1}}}-\frac{1}{k^{\frac{k}{k-1}}}\right)
\frac{l(p+2)^{\frac{1}{k-1}}(2n-mq)^{\frac{k}{k-1}}}{2n(pm+2m-2n)^{\frac{1}{k-1}}q^{\frac{k}{k-1}}}>0$. Hence $J_{\la}$ is bounded below on $\mc N_{\la}.$ \QED

\begin{Lemma}\label{le31}
Let $\la$ satisfy (\ref{eq8}).
Then given $u\in \mc N_{\la}\setminus \{0\}$, there exist $\e>0$
and a differentiable function $\xi : \textbf{B}(0,\e)\subset
W^{m,\frac{n}{m}}_{0}(\Om) \lra \mb{R}$ such that $\xi(0)=1$, the function
$\xi(w)(u-w)\in\mc N _{\la}$ and for all $w\in W^{m,\frac{n}{m}}_{0}(\Om)$
\begin{align}\label{eq:3.1}
\ld&\xi^{\prime}(0),w\rd =  \notag\\
&\frac{ \frac{n}{m}(a+2b\|u\|^\frac{n}{m}) \displaystyle \int_{\Om}|\nabla^m u|^{\frac{n}{m}-2}\na^{m} u
\na^{m} w -\ds\int_{\Om}\frac{\left({f(u)+f^{\prime}(u)u}\right)w}{|x|^\alpha}dx - \la
q\ds\int_{\Om} h(x)|u|^{q-2}uw dx}{a\left(\frac{n}{m}-q\right)\|u\|^\frac{n}{m}+b\left(\frac{2n}{m}-q\right)\|u\|^{\frac{2n}{m}}-(1-q)\int_{\Om} \frac{f(u)u}{|x|^\alpha}-\int_{\Om} \frac{f^\prime(u)u^2}{|x|^\alpha}dx}.
\end{align}
\end{Lemma}
\proof~ Fix $u\in \mc N_{\la}\setminus \{0\}$, define a function
$F_u: \mb R\times W^{m,\frac{n}{m}}_{0}(\Om) \lra \mb R$ as follows:
\begin{align*}
F_{u}(t,v)= at^{\frac{n}{m}-q}\|u-v\|^\frac{n}{m}+bt^{\frac{2n}{m}-q}\|(u-v)\|^{\frac{2n}{m}}&-  t^{1-q}\int_{\Om} \frac{f(t(u-v)) (u-v)}{|x|^\alpha}dx\\ &-\la \int_{\Om}h(x)|u-v|^{q}dx.
\end{align*}
 Then $F_u\in C^{1}(\mb R\times W^{m,\frac{n}{m}}_{0}(\Om); \mb
R)$, $F_{u}(1,0)= \ld J_{\la}^{\prime}(u),u \rd= 0$ and
\begin{equation*}
\frac{\partial}{\partial t}F_{u}(1,0)= a\left(\frac{n}{m}-q\right)\|u\|^\frac{n}{m}+b\left(\frac{2n}{m}-q\right)\|u\|^{\frac{2n}{m}}-(1-q)\int_{\Om} \frac{f(u)u}{|x|^\alpha}-\int_{\Om} \frac{f^\prime(u)u^2}{|x|^\alpha}dx,
\end{equation*}
since $\mc N_{\la}^{0}=\{0\}$. By the Implicit function
theorem, there exist $\e>0$ and a differentiable function $\xi :
\textbf{B}(0,\e)\subset W^{m,\frac{n}{m}}_{0}(\Om) \lra \mb{R}$ such that
$\xi(0)=1$ and $F_{u}(\xi(w),w)=0$ for all $w\in \textbf{B}(0,\e)$
which is equivalent to $\ld
J_{\la}^{\prime}(\xi(w)(u-w)),\xi(w)(u-w) \rd=0$ for all $w\in
\textbf{B}(0,\e)$. Hence $\xi(w)(u-w)\in \mc N _{\la}$. Now
differentiating $F_{u}(\xi(w),w)=0$ with respect to $w$ we obtain
\eqref{eq:3.1}. \QED
\begin{Lemma}\label{le33}
 There exists a constant $C_2>0$ such that $\theta_{\la}\leq
-\left(\frac{p+2-q}{2nq(p+2)}\right)C_2$.
\end{Lemma}
\proof~ Let $v$ be such that $\int_{\Om} h|v|^{q} ~dx>0$. Then by the Lemma \ref{le5}, we find $t_{1}=t_{1}(v)>0$ such that $t_{1}v\in \mc
N_{\la}^+$. Thus
\begin{align}\label{eq31}
&J_{\la}(t_{1}v)=\left(\frac{m}{n}-\frac{1}{q}\right)a\|t_{1}v\|^{\frac{n}{m}}+\left(\frac{m}{2n}-\frac{1}{q}\right)b\|t_{1}v\|^{\frac{2n}{m}} -
\int_{\Om}\frac{F(t_{1}v)}{|x|^\alpha}~dx +\frac{1}{q} \int_{\Om} \frac{f(t_{1}v) t_{1}v}{|x|^\alpha} ~dx\notag\\
&\leq \left(\frac{mq+2n-m}{2nq}\right)\int_{\Om}\frac{f(t_{1}v) t_{1}v}{|x|^\alpha}~dx -\int_{\Om}
\frac{F(t_{1}v)}{|x|^\alpha}~dx -\frac{m}{2nq} \int_{\Om} \frac{f^{\prime}(t_{1}v) (t_{1}v)^2}{|x|^\alpha}~dx,
\end{align}
since $t_{1}v\in \mc N_{\la}^+$. Now consider the
function
\begin{equation*}
\rho(s)= \left(\frac{mq+2n-m}{2nq}\right)f(s)s- F(s)-\frac{m}{2nq}f^{\prime}(s)s^2.
\end{equation*}
Then%
\begin{equation*}
\rho^{\prime}(s)=
\left(\frac{mq+2n-3m}{2nq}\right)f^{\prime}(s)s-\left(\frac{mq+2n-m-2nq}{2nq}\right)f(s)-\frac{m}{2nq}f^{\prime\prime}(s)s^2
\end{equation*}
Now it can be easily seen that $\rho^\prime<0$ for all $s\in\mb R^{+}$. Also it is easy to verify that
\[\lim_{s\ra 0}\frac{\rho(s)}{|s|^{p+2}}=
-\frac{(p+2-q)(mp+2m-2n)}{2nq(p+2)}\; \;\text{and}\;\;
\lim_{s\ra \infty}\frac{\rho(s)}{|s|^{p+2+\ba} e^{|s|^{\ba}}}=
-\frac{m\ba}{2nq}.\]
From these two estimates, we get that
\begin{equation}\label{eq30}
\rho(s)\leq-\frac{p+2-q}{2nq(p+2)}\left(mp+2m-2n+m\ba|s|^{\ba}\right)
|s|^{p+2}e^{|s|^{\ba}}.
\end{equation}
Therefore, using (\ref{eq31}), (\ref{eq30}) and $\frac{2n}{m}<p+2$, we get
\begin{align*}
J_{\la}(t_{1}v)&\leq - \frac{(p+2-q)}{2nq(p+2)}\int_{\Om}
\left(mp+2m-2n+m\ba|t_{1}v|^{\ba}\right) |t_{1}v|^{p+2} ~
\frac{e^{|t_{1}v|^{\ba}}}{|x|^\alpha} ~dx \notag\\
& \leq - \frac{m(p+2-q)}{2nq(p+2)}\int_{\Om}\frac{|t_{1}v|^{p+2+\ba}}{|x|^\alpha} ~dx.
\end{align*}
Hence $\ds \theta_{\la}\leq\inf_{u\in \mc N_{\la}^{+}}J_{\la}(u)\leq  -\frac{m(p+2-q)}{2nq(p+2)}\; C_2$, where
$\ds C_2=\int_{\Om} \frac{|t_{1}v|^{p+2+\ba}}{|x|^\alpha}dx.$\QED

\noi Now from Lemma \ref{th1},  $J_{\la}$ is bounded below on $\mc
N_{\la}$. So, by Ekeland's Variational principle, we can find a
sequence $\{u_k\}\subset \mc N_{\la}\setminus \{0\}$ such that
\begin{equation}\label{eq32}
J_{\la}(u_k)\leq \theta_{\la}+\frac{1}{k},
\end{equation}
\begin{equation}\label{eq33}
J_{\la}(v)\geq J_{\la}(u_k)- \frac{1}{k}\|v-u_k\|\;\;\text{for all}\;\;v\in \mc
N_{\la}.
\end{equation}

\noi Now from (\ref{eq32}) and Lemma \ref{le33},
we have
\begin{align}\label{eq34}
 J_{\la}(u_k)&\leq -\frac{m(p+2-q)}{2nq(p+2)}\; C_2+\frac{1}{k}.
\end{align}
\noi As $u_k\in \mc N_{\la}$, we have
\begin{align*}
J_{\la}(u_k) = &\left(\frac{m}{n}-\frac{1}{p+2}\right) a\|u_k\|^{\frac{n}{m}}+\left(\frac{m}{2n}-\frac{1}{p+2}\right)b\|u_k\|^{\frac{2n}{m}}
-\lambda\left(\frac{1}{q}-\frac{1}{p+2}\right) \int_{\Om} h |u_k|^{q}~dx\\
&\;\;+\int_{\Om}\frac{\left(\frac{f(u_k)u_k}{p+2} - F(u_k)\right)}{|x|^\alpha}~dx.
\end{align*}
This together with  \eqref{eq34} and $\frac{1}{p+2}f(u_k)u_k - F(u_k)\geq 0$, we can get $C_3=\frac{m}{2n\lambda}>0$ such that
\begin{align}\label{eq35}
\int_{\Om} h(x)|u_k|^{q} ~dx\geq C_3 >0.
\end{align}

\noi Now we
prove the following:
\begin{Proposition}\label{pro1}
Let $\la$ satisfy (\ref{eq8}).
Then $\|J_{\la}^{\prime}(u_k)\|_{*}\ra 0$ as $k\ra \infty$.
\end{Proposition}
\proof~ Step 1: $\ds\liminf_{k\ra \infty} \|u_k\|>0$.\\
\noi Applying H\"{o}lders inequality in (\ref{eq35}), we have
$K'\|u_k\|^{q}\geq \int_{\Om}h |u_k|^{q}~dx \geq C_3>0$ which implies
that $\ds \liminf_{k\ra \infty} \|u_k\|>0$.\\
Step 2: We claim that
\begin{equation}\label{liminf3.4}
\ds K:=\liminf_{k \ra \infty}\left\{a\left(\frac{n}{m}-q\right)\|u\|^\frac{n}{m}+b\left(\frac{2n}{m}-q\right)\|u\|^{\frac{2n}{m}}-\int_{\Om} \frac{(1-q)f(u)u-f^\prime(u)u^2}{|x|^\alpha}dx\right\}>0.
\end{equation}
Assume by contradiction that for some subsequence of $\{u_k\}$,
still denoted by $\{u_k\}$ we have
\begin{align*}
a\left(\frac{n}{m}-q\right)\|u_k\|^\frac{n}{m}+b\left(\frac{2n}{m}-q\right)\|u_k\|^{\frac{2n}{m}}-(1-q)\int_{\Om} \frac{f(u_k)u_k}{|x|^\alpha}-\int_{\Om} \frac{f^\prime(u_k)u_k^2}{|x|^\alpha}dx= o_{k}(1),
\end{align*}
where $o_k(1)\ra 0$ as $k\ra \infty.$\\
From this and the fact that the sequence $\{u_k\}$ is bounded away from $0$,
we obtain that\\ $\ds \liminf_{k\ra \infty}\int_{\Om}\frac{f^{\prime}(u_k) u_k^2}{|x|^\alpha}~
dx >0.$ Hence, we get $u_k \in \La\setminus \{0\}$ for all $k$ large.
Using this and the fact that $u_k \in \mc N_{\la}\setminus \{0\}$,
we have
\begin{equation*}
o_{k}(1)= \la\left(\frac{2n}{m}-q\right) \int_{\Om}h |u_k|^{q} ~dx-\int_{\Om}\frac{\left(f^{\prime}(u_k) u_k^2-(\frac{2n}{m}-1)f(u_k)u_k\right)}{|x|^\alpha} ~dx-a\frac{n}{m}\|u_k\|^\frac{n}{m} < - I_\la
\end{equation*}
by (\ref{eq8}), which is a contradiction.

\noi Finally, we show that $\|J_{\la}^\prime(u_k)\|_{*}\ra 0$ as $k\ra
\infty$. By Lemma \ref{le31}, we obtain a sequence of functions
$\xi_k :\textbf{B}(0,\e_k)\ra \mb R$ for some $\e_k>0$ such that
$\xi_k(0)=1$ and $\xi_k(w)(u_k-w)\in \mc N_{\la}$ for  all
$w\in\textbf{B}(0,\e_k)$. Choose $0<\rho <\e_k$ and $f\in
W^{m,\frac{n}{m}}_{0}(\Om) $ such that $\|f\|=1$. Let $w_\rho=\rho f$. Then
$\|w_{\rho}\|=\rho<\e_k$ and
$\eta_{\rho}=\xi_k(w_\rho)(u_k-w_\rho)\in\mc N_{\la}$ for all $k$.
Since $\eta_{\rho} \in\mc N_{\la}$, we deduce from (\ref{eq33}) and
Taylor's expansion,
\begin{align}\label{eq:3.7}
\frac{1}{k}\|\eta_{\rho}-u_k\|&\geq J_{\la}(u_k)-J_{\la}(\eta_\rho)
=\ld J_{\la}^{\prime}(\eta_\rho),u_k-\eta_\rho\rd + o(\|u_k-\eta_\rho\|)\notag\\
&= (1-\xi_k(w_{\rho}))\ld J_{\la}^{\prime}(\eta_\rho),u_k\rd + \rho
\xi_k(w_{\rho})\ld J_{\la}^{\prime}(\eta_\rho),f\rd +
o(\|u_k-\eta_\rho\|).
\end{align}
We note that as $\rho \ra 0$, $\frac{1}{\rho}\|\eta_\rho -u_k\|=
\|u_k\ld\xi_{k}^{\prime}(0),f\rd -f\|.$ Now dividing \eqref{eq:3.7}
by $\rho$ and taking limit $\rho \ra 0$, by Lemma \ref{le31} and \eqref{liminf3.4}, we get
\begin{equation*}
\ld J_{\la}^{\prime}(u_k),f \rd\leq
\frac{1}{k}\left(\|u_k\|\|\xi^{\prime}_{k}(0)\|_{*}+1\right)\le
\frac{1}{k}\frac{C_4\|f\|}{K},
\end{equation*}
This completes the proof
of Proposition.\QED

\subsection{Existence of a local minimum of $J_{\la}$ on $\mc N_{\la}$, $\la\in(0,\la_0)$ }

\begin{Lemma}\label{zle34}
Let $\ba<\frac{n}{n-m}$ and let $\la$ satisfy (\ref{eq8}). Then there exists a function $u_{\la}\in \mc N_\la^{+}$ such that $\ds J_{\la}(u_{\la})=\inf_{u\in \mc N_{\la}\setminus\{0\}}J_{\la}(u)$.
\end{Lemma}
\proof Let $\{u_k\}$ be a minimizing sequence for $J_{\la}$ on $\mc
N_{\la}\setminus \{0\}$ satisfying (\ref{eq32}) and (\ref{eq33}).
Then $\{u_k\}$ is bounded in $W_{0}^{m,\frac{n}{m}}(\Om)$. Also there exists a subsequence of $\{u_k\}$ (still denoted by
$\{u_k\}$) and a function $u_\la$ such that $u_k \rightharpoonup
u_\la$ weakly in $W_{0}^{m,\frac{n}{m}}(\Om)$, $u_k\ra u_\la$ strongly in
$L^{\al}(\Om)$ for all $\al\geq 1$ and $u_k(x)\ra u_{\la}(x)$ a.e in
$\Om$. Also $\int_{\Om} h|u_k|^{q}dx \ra \int_{\Om} h|u_\la|^{q}dx$ and by the compactness of Moser-Trudinger imbedding for $\beta<\frac{n}{n-m}$, $\int_{\Om} \frac{f(u_k)(u_k-u_\la)}{|x|^\alpha}dx\ra 0$ as $k\ra\infty$. Then by Lemma \ref{pro1}, we have $\langle J_{\la}^{\prime}(u_k),(u_k-u_\la)\rangle \ra 0$. We conclude that
 \[ M(\|u_k\|^\frac{n}{m})\int_{\Om}|\na^m u_k|^{\frac{n}{m}-2}\nabla ^mu_k(\na^m u_k-\na^m u_\la)\ra 0.\]
On the other hand, using $u_k\rightharpoonup u_0$ weakly and boundedness of $M(\|u_k\|^\frac{n}{m})$,
{\small\[M(\|u_k\|^\frac{n}{m}) \int_{\Om} |\na^m u_\la|^{\frac{n}{m}-2} \na^m u_\la (\na^m u_k - \na^m u_\la)\ra 0\; \mbox{as}\; k\ra\infty.\]}
From above two equations and inequality \begin{align*}
|a-b|^l\leq 2^{l-2}(|a|^{l-2}a-|b|^{l-2}b)(a-b)\;\mbox{for all}\; a,b\in \mb R^n\;\mbox{and}\; l\geq 2,
\end{align*} we have
\[M(\|u_k\|^\frac{n}{m})\int_{\Om}|\na^m u_k- \na^m u_\la|^\frac{n}{m} \ra 0\;\mbox{as}\; k\ra\infty.\]
Since $M(t)\ge M_0$, we obtain $u_k\ra u$ strongly in $W^{m,\frac{n}{m}}_{0}(\Om)$ and hence $\|u_k\| \rightarrow \|u_0\|$ strongly as $k\ra\infty$. In particular, it follows that $u_{\la}$ solves $(\mathcal P_{\la})$ and hence $u_\la \in \mc N_{\la}$. Moreover,
$\ds\theta_\la \leq J_{\la}(u_\la)\leq \liminf_{k\ra
\infty}J_{\la}(u_k)=\theta_\la$. Hence $u_\la$ is a minimizer for
$J_{\la}$ on $\mc N_{\la}$.\\
\noi Using (\ref{eq35}), we have $\int_{\Om} h|u_\la|^{q}>0$. Therefore there exists $t_{1}(u_\la)$ such that
$t_{1}(u_{\la})u_\la \in \mc N_{\la}^{+}$. We now claim that
$t_{1}(u_\la)=1$ $(i.e.$ $u_\la \in \mc N_{\la}^{+})$. Suppose
$t_{1}(u_\la)<1$. Then $t_{2}(u_\la)=1$ and hence $u_\la \in \mc
N_{\la}^{-}$. Now $J_{\la}(t_{1}(u_\la)u_\la)\leq J_{\la}(u_\la)=
\theta_{\la}$ which is impossible, as $t_{1}(u_\la) u_\la \in \mc
N_{\la}$.\QED

\begin{Theorem}\label{zth3.6}
Let $\ba<\frac{n}{n-m}$ and let $\la$ be such that (\ref{eq8}) holds.
Then $u_\la \in \mc N_{\la}^{+}$ is also a
non-negative local minimum for $J_{\la}$ in $W^{m,\frac{n}{m}}_{0}(\Om)$.
\end{Theorem}
\proof Since $u_\la \in \mc N^{+}_{\la}$, we have $t_{1}(u_{\la})=1
<t_*(u_{\la})$. Hence by continuity of $u\mapsto t_*(u)$, given
$\e>0$, there exists $\de=\de(\e)>0$ such that $1+\e< t_*(u_\la-w)$
for all $\|w\|<\de$. Also, from Lemma \ref{le33} we have, for $\de>0$
small enough, we obtain a $C^1$ map $t: \textbf{B}(0,\de)\lra \mb R^+$
such that $t(w)(u_\la-w)\in \mc N_{\la}$, $t(0)=1$. Therefore, for
$\de>0$ small enough we have $t_{1}(u_\la-w)=
t(w)<1+\e<t_*(u_\la-w)$ for all $\|w\|<\de$. Since $t_*(u_\la-w)>1$,
we obtain $J_{\la}(u_\la)<J_{\la}(t_{1}(u_\la-w)(u_\la-w))<J_{\la}(u_\la-w)$
for all $\|w\|<\de$. This shows that $u_\la$ is a local minimizer for
$J_{\la}$.\\
\noi Now we show that $u_\la$ is a non-negative local minimum for
$J_{\la}$ on $W^{m,\frac{n}{m}}_{0}(\Om)$. If $u_\la\geq 0$ then we are done,
otherwise, if $u_{\la}\not\geq 0$ then we take $\tilde{u_{\la}}=
t_1(|u_\la|)|u_{\la}|$ which is non negative function in $\mc
N_{\la}^{+}$. As $\psi_{u_{\la}}(t)= \psi_{|u_{\la}|}(t)$ so
$t_{*}(|u_{\la}|)=t_{*}(u_{\la})$ and $t_{1}(u_{\la})\leq
t_1(|u_\la|)$. Hence $ t_1(|u_\la|)\geq 1$. Then from Lemma \ref{le5} we have $J_{\la}(\tilde{u_{\la}}) \leq
J_{\la}(|u_{\la}|)\leq J_{\la}(u_{\la})$. Hence $\tilde{u_{\la}}$
minimize $J_{\la}$ on $\mc N_{\la}\setminus \{0\}$. Thus we can
proceed same as earlier to show that $\tilde{u_{\la}}$ is a local
minimum for $J_{\la}$ on $W^{m,\frac{n}{m}}_{0}(\Om).$\QED

\begin{Lemma}\label{zle03} Let $\beta<\frac{n}{n-m}$ and let $\la$ be such that \eqref{eq8} holds. Then
$J_{\la}$ achieve its minimizers on $\mc N_{\la}^{-}$.
\end{Lemma}
\proof
We note that $\mc N_{\la}^{-}$ is a closed set, as $t^{-}(u)$
is a continuous function of $u$ and  ${J}_{\la}$ is bounded
below on $\mc N_{\la}^{-}$. Therefore, by Ekeland's Variational
principle, we can find a sequence $\{v_k\}\in \mc N_{\la}^{-}$ such
that
{\small\begin{align*}
{J}_{\la}(v_k)&\leq \inf_{u\in \mc
N_{\la}^{-}}{J}_{\la}(u) +\frac{1}{k},\;\;
{J}_{\la}(v)\geq{J}_{\la}(v_k)-\frac{1}{k}\|v-v_k\|
\;\mbox{for all}\; v\in \mc N_{\la}^{-}.
\end{align*}}
Then $\{v_k\}$ is a bounded sequence in $W^{m,\frac{n}{m}}_{0}(\Om)$ and
is easy to see that $v_k\in
\La\setminus \{0\}$. Thus by following the proof of Lemma \ref{pro1}, we get
$\|{J}_{\la}^{\prime}(v_k)\|_{*}\ra 0$ as $k\ra \infty$. Thus
following the proof as in Lemma \ref{zle34}, we have $v_\la\in\mc N_{\la}^{-}$, weak
limit of sequence $\{v_k\}$, is a solution of $(\mathcal P_{\la})$. And moreover $v_\la \not\equiv 0$, as $\mc
N^{0}_{\la}=\{0\}$.\QED

\noi{\bf Proof of Theorem \ref{tp2}: } Now the proof follows from Lemmas \ref{zle34} and \ref{zle03}.\QED

\noi To obtain the existence result in the critical case, we need the following compactness Lemma.
\begin{Lemma}\label{cpt}
Suppose $\{u_k\}$ be a sequence in $W^{m,\frac{n}{m}}_{0}(\Om)$ such that
\[ J_{\la}^{\prime}(u_k) \rightarrow 0\;\; J_{\la}(u_k) \rightarrow c < \frac{m}{2n}M_0\beta_{\al}^\frac{n-m}{m}- C \la^{\frac{p+2+\ba}{p+2-q+\ba}},\] where $C$ is a positive constant depending on $p$, $q$, $m$ and $n$.
Then there exists a strongly convergent subsequence.
\end{Lemma}
\proof
By Lemma \ref{bps}, there exists a subsequence of $\{u_k\}$, (say $\{u_k\}$ ) such that $u_k \rightarrow u$ in $L^\al(\Om)$ for all $\al$, $u_k(x) \rightarrow u(x)$ a.e. in $\Om$, $\na^m u_k(x) \rightarrow \na^m u(x)$ a.e. in $\Om$. As $\{u_k\}$ is a Palais- Smale sequence,
$|\na^m u_k|^{\frac{n}{m}-2}\na^m u_k$ is bounded in $L^\frac{n}{n-m}$ and $\frac{f(u_k)u_k}{|x|^\alpha}$ is bounded  $L^1(\Omega)$. So by concentration compactness lemma, $|\na^m u_k|^\frac{n}{m} \rightarrow \mu_1$, $\frac{f(u_k)u_k}{|x|^\alpha} \rightarrow \mu_2$ in measure.

\noi Let $B=\{ x\in \overline{\Om}: \; \exists\; r=r(x),\;  \mu_1(\textbf{B}_r\cap \Omega) < \beta_{\alpha}^\frac{n-m}{n}\}$ and let $A=\overline{\Om}\backslash B.$ Then as in Lemma \ref{IC2}, we can show that $A$ is finite set say $\{x_1, x_2,...x_m\}.$ Since $J_{\la}^{\prime} (u_k)\rightarrow 0$, we have
{\small\begin{align}
0=&\lim_{k\ra\infty} \langle J^{\prime}_{\la}(u_k), \phi\rangle = \lim_{k\ra\infty} M(\|u_k\|^\frac{n}{m}) \int_\Om |\na^m u_k|^{\frac{n}{m}-2} \na^m u_k \na^m \phi -\la \int_\Om h|u_k|^{q-2} u_k \phi dx\nonumber\\&\hspace{9cm}-\int_\Om \frac{f(u_k)}{|x|^\alpha} \phi dx\nonumber\\
0=&\lim_{k\ra\infty}\langle J^{\prime}_{\la}(u_k), u_k \phi\rangle = \lim_{k\ra\infty}M(\|u_k\|^\frac{n}{m}) \left(\sum_{l=1}^m\int_\Om |\na^m u_k|^{\frac{n}{m}-2} \na^m u_k \na^l \phi \nabla ^{m-l}u_k + \int_\Om |\na^m u_k|^\frac{n}{m}  \phi \right) \notag\\ &\hspace*{4cm}-\la\int_\Om h|u_k|^{q}\phi- \int_\Om \frac{f(u_k) u_k}{|x|^\alpha} \phi dx \label{a2}\\
0=&\lim_{k\ra\infty} \langle J^{\prime}_{\la} (u_k), u\phi\rangle= \lim_{k\ra\infty} M(\|u_k\|^\frac{n}{m})\left(\sum_{l=1}^m\int_\Om |\na^m u_k|^{\frac{n}{m}-2} \na^m u_k \na^{m-l} u\nabla^l\phi\right)\nonumber\\&+\lim_{k\ra\infty} M(\|u_k\|^\frac{n}{m})\left(\int_\Om |\na^m u_k|^{\frac{n}{m}-2}\na^m u_k \na^m u\phi\right)-\la \int_\Om h|u_k|^{q-2}u_ku \phi -\int_{\Om} \frac{f(u_k) u}{|x|^\alpha} \phi \label{a3}
\end{align}}
Using \eqref{usela}, \eqref{a3} in \eqref{a2} , we have
{\small \begin{equation}\label{a4}
\lim_{k\rightarrow \infty}\int_\Om \frac{f(u_k) u_k}{|x|^\alpha} \phi = \lim _{k\rightarrow \infty }M(\|u_k\|^\frac{n}{m})\int_\Om( |\na^m u_k|^\frac{n}{m} -  |\na^m u_k|^{\frac{n}{m}-2} \na^m u_k \na^m u) \phi +\int_\Om \frac{f(u) u}{|x|^\alpha} \phi
\end{equation}}
Now take cut-off function $\psi_\de \in C^{\infty}_{0}(\Om)$ such that $\psi_\de(x)\equiv1$ in $\textbf{B}_\de(x_j),$ and $\psi_\de (x)\equiv0$ in $\textbf{B}_{2\de}^{c}(x_j)$ with $|\psi_\de|\le 1$. Then taking $\phi=\psi_\de$,
{\small\[0\le \left|\int_\Om |\na^m u_k|^{\frac{n}{m}-2} \na^m u_k \na^m u \phi\right| \le \left(\int_{\Om} |\na^m u_k|^\frac{n}{m}\right)^{(n-m)/n} \left(\int_{\textbf{B}_{2\de}} |\na^m u|^\frac{n}{m}\right)^{m/n} \rightarrow 0 \;\; \text{as}\;\; \de \rightarrow 0.\]}
Hence from \eqref{a4}, we get
\begin{equation}\label{a5}
\int_\Om \phi d \mu_2 \ge M_0 \int_\Om \phi d\mu_1 + \int_\Om
\frac{f(u)u}{|x|^\alpha} \phi\;\mbox{as}\; \de\ra 0.
\end{equation}
Now as in Assertion 1 of Lemma \ref{IC2}, we can show that for any relatively compact set $K\subset \Om_\e$, where $\Om_\e = \Om\backslash \cup_{i=1}^{m} \textbf{B}_\de(x_i)$
\[\lim_{k\rightarrow \infty} \int_K \frac{f(u_k)u_k}{|x|^\alpha} \rightarrow \int_K \frac{f(u) u}{|x|^\alpha}.\]
Also taking  $0<\e<\e_0$ and $\phi\in C^\infty_c(\mb R^n)$ such that
$\phi\equiv1$ in $\textbf{B}_{1/2}(0)$ and $\phi\equiv0$ in $\bar\Om\setminus \textbf{B}_1(0)$.
Take {$\psi_\e=\ds\sum_{j=1}^r\phi\left(\frac{x-x_j}{\e}\right)$ in \eqref{a5}. Then $0\leq\psi_\e\leq 1$, $\psi_\e\equiv 0$ in
 $\bar\Om_\e=\bar\Om\setminus\cup_{j=1}^r \textbf{B}_\e(x_j)$,$\psi_\e\equiv 1$ in $\cup_{j=1}^r \textbf{B}_{\e/2}(x_j)$}
{\small\begin{align*}
\int_\Om \psi_\e d \mu_2 & =\lim_{\e\rightarrow 0}\left(\int_{\Om_\e} \psi_\e d \mu_2 + \sum_{i=1}^{r} \int_{\textbf{B}_{\e}\cap \Om} \psi_\e d\mu_2\right)
=\lim_{\e \rightarrow 0}\int_{\Om_\e} \frac{f(u) u}{|x|^\alpha} \psi_\e + \sum_{i=1}^{r} \ba_i \de_{x_i}\\
&=\int_\Om \frac{f(u) u}{|x|^\alpha} +\sum_{i=1}^{r} \ba_i \de_{x_i}.
\end{align*}}
Therefore, from \eqref{a5}, we get
\begin{equation*}
M_0 \int_\Om \psi_\e d\mu_1 \le \sum_{i=1}^{r} \ba_i \de_{x_i}.
\end{equation*}
Now choosing $\e \rightarrow 0$, we get
\[M_0\mu_1(A) \le \sum_{i=1}^{r} \ba_i .\]
Therefore from the definition of $A$, either $\ba_i=0$ or $\ba_i \ge M_0 \beta_{\alpha}^\frac{n-m}{m}$.
Now we will show that $\ba_i=0,$ for all $i$. Suppose not,
Now using $J_{\la}(u_k)\rightarrow c$ and \eqref{reqla}, we have
{\small\begin{align*}
c= &J_\la(u_k)- \frac{m}{2n} \langle J_{\la}^{\prime} (u_k), u_k\rangle \\
 = &\left(\widehat M(\|u_k\|^\frac{n}{m})-\frac{m}{2n} M(\|u_k\|^n)\|u_k\|^\frac{n}{m}\right)+ \int_\Om \frac{\left(\frac {m}{2n} f(u_k) u_k - F(u_k)\right)}{|x|^\alpha}
 \\
 &+\la\left(\frac{m}{2n}-\frac{1}{q}\right)\int_\Om h |u|^{q}dx\\
\ge & \frac{ M_0\beta_{\al}^\frac{n-m}{m}}{2n}+\int_\Om \frac{\left(\frac{m}{2n} f(u)u-F(u)\right)}{|x|^\alpha}dx + \la\left(\frac{m}{2n}-\frac{1}{q}\right)   \int_\Om h|u|^{q}dx.
\end{align*}}
\noi Therefore,
\begin{align*}
c\ge &\frac{m}{2n}M_0\beta_{\al}^\frac{n-m}{m} + \left(\frac{m}{2n}-\frac{1}{p+2}\right)\int_{\Om} \frac{|u|^{p+2+\ba}}{|x|^\alpha}dx + \la \left(\frac{m}{2n}-\frac{1}{q}\right)\int_\Om h |u|^{q}dx\\
\geq & \frac{m}{2n}M_0\beta_{\al}^\frac{n-m}{m} + \left(\frac{m}{2n}-\frac{1}{p+2}\right)\int_{\Om} \frac{|u|^{qk}}{|x|^\alpha}dx+ \la\left(\frac{m}{2n}-\frac{1}{q}\right)l^\frac{k-1}{k}\left(\int_\Om  \frac{|u|^{qk}}{|x|^\alpha}dx \right )^{\frac{1}{k}},
\end{align*}
where $k=\frac{p+2+\ba}{q}$. Now as in Theorem \ref{th1}, consider the global minimum of the function $\rho(y): \mb R^+
\lra \mb R$ defines as $$\rho(y)=\left(\frac{mp+2m-2n}{2n(p+2)}\right)y^k -
\frac{\la(2n-mq)l^{\frac{k-1}{k}}}{2nq} y.$$
Then $\rho(y)$ attains its minimum at $\left(\frac{\la(2n-mq)(p+2)l^{\frac{k-1}{k}}}{kq(mp+2m-2n)}\right)^{\frac{1}{k-1}}$. So,
and its minimum value is  {\small$-C(p,q,n,m)\la^{\frac{k}{k-1}}$, where $C(p,q,n,m)=
\left(\frac{1}{k^{\frac{1}{k-1}}}-\frac{1}{k^{\frac{k}{k-1}}}\right)
\frac{l(p+2)^{\frac{1}{k-1}}(2n-mq)^{\frac{k}{k-1}}}{n(pm+2m-2n)^{\frac{1}{k-1}}(q)^{\frac{k}{k-1}}}>0.$
Therefore, $c\ge \frac{m}{2n}M_0\beta_{\al}^\frac{n-m}{m} - C(p,q,n,m)\la^{\frac{p+2+\ba}{p+2-q+\ba}}$ .\QED

\noi Let $\la_{00}=\max \{ \la: \; \theta_{\la} \le \frac{m}{2n}M_0\beta_{\al}^\frac{n-m}{m}- C \la^{\frac{p+2+\ba}{p+2-q+\ba}}\}$ where $C=C(p,q,n,m)$ is as in the above Lemma.\\

\noi{\bf Proof of Theorem \ref{tp1}:} Let $\{u_k\}$ be a minimizing sequence for $J_{\la}$ on $\mc
N_{\la}\setminus \{0\}$ satisfying (\ref{eq33}) {and $\lambda^0=\min\{\lambda_0, \lambda_{00}\}$.}
Then it is easy to see that $\{u_k\}$ is a bounded sequence in $W_{0}^{m,\frac{n}{m}}(\Om)$.
Also there exists a subsequence of $\{u_k\}$ (still denoted by
$\{u_k\}$) and a function $u_\la$ such that $u_k \rightharpoonup
u_\la$ weakly in $W_{0}^{m,\frac{n}{m}}(\Om)$, $u_k\ra u_\la$ strongly in
$L^{\al}(\Om)$ for all $\al\geq 1$ and $u_k(x)\ra u_{\la}(x)$ a.e in
$\Om$. Then by Lemma \ref{pro1}, we have $\langle J_{\la}^{\prime}(u_k),(u_k-u_\la)\rangle \ra 0$.

\noi Now by compactness Lemma \ref{cpt}, $u_k \rightarrow u_\la$ strongly in $W^{m,\frac{n}{m}}_{0}(\Om)$ and hence $\|u_k\|\rightarrow \|u_\la\|$ strongly as $k\rightarrow \infty$. In particular, it follows that $u_\la$ solves $(\mathcal P_{\la})$ and hence $u_\la \in \mc N_{\la}.$ Also we can show similarly as in Lemma \ref{zle34} and Theorem \ref{zth3.6} that $u_{\la} \in \mc N_{\la}^{+}$ is a non-negative local minimizer of $J_{\la}$ in $W^{m,\frac{n}{m}}_{0}(\Om)$.\QED

\end{document}